\documentclass[11pt,leqno]{article}
\usepackage{amssymb,amsfonts}
\usepackage{amsmath,amsthm,amsxtra}
\usepackage[all]{xy}
\usepackage{color}
\usepackage{verbatim}
\usepackage{mathtools}

\usepackage{enumerate}

\newcommand\bigcheck[1]{#1 \raise1ex\hbox{$\hspace{-1ex}{}^\vee$}}
\newcommand\sucheck[1]{#1 \raise0.5ex\hbox{$\hspace{-1ex}{}^\vee$}}

\newtheorem{theorem}{Theorem}[section]
\newtheorem{lemma}[theorem]{Lemma}
\newtheorem{corollary}[theorem]{Corollary}
\newtheorem{proposition}[theorem]{Proposition}
\newtheorem*{lemma*}{Lemma}

\theoremstyle{definition}
\newtheorem{definition}[theorem]{Definition}
\newtheorem{defprop}[theorem]{Definition/Proposition}

\theoremstyle{remark}
\newtheorem{remark}[theorem]{Remark}

\newcommand{\mc}[1]{{\mathcal #1}}

\newcommand{\mb}[1]{{\mathbb #1}}

\newcommand{\id}{{1 \mskip -5mu {\rm I}}}

\renewcommand{\tilde}{\widetilde}

\newcommand{\Mat}{\mathop{\rm Mat }}

\renewcommand{\ker}{\mathop{\rm Ker }}
\newcommand{\im}{\mathop{\rm Im }}

\newcommand{\sdeg}{\mathop{\rm sdeg }}

\newcommand{\ass}[1]{\xleftrightarrow[\phantom{ciao}]{#1}}
\newcommand{\assk}[2]{\xleftrightarrow[#2]{#1}}

\definecolor{light}{gray}{.9}

\begin{document}


\title{Singular degree of a rational matrix pseudodifferential operator}

\author{
Sylvain Carpentier
\thanks{Ecole Normale Superieure, 
75005 Paris, France, and M.I.T~~
sylvain.carpentier@ens.fr ~~~~},~~
Alberto De Sole
\thanks{Dipartimento di Matematica, Universit\`a di Roma ``La Sapienza'',
00185 Roma, Italy ~~
desole@mat.uniroma1.it 
},~~
Victor G. Kac
\thanks{Department of Mathematics, M.I.T.,
Cambridge, MA 02139, USA.~~
kac@math.mit.edu~~~~
Supported in part by 
the Simons Fellowship~~ 
}~~
}

\maketitle


\begin{abstract}
\noindent 
In our previous work we studied minimal fractional decompositions
of a rational matrix pseudodifferential operator:
$H=AB^{-1}$, where $A$ and $B$ are matrix differential operators,
and $B$ is non-degenerate of minimal possible degree $\deg(B)$.
In the present paper we introduce the singular degree $\sdeg(H)=\deg(B)$,
and show that for an arbitrary rational expression
$H=\sum_{\alpha}A^\alpha_1(B^\alpha_1)^{-1}\dots A^\alpha_n(B^\alpha_n)^{-1}$,
we have $\sdeg(H)\leq\sum_{\alpha,i}\deg(B^\alpha_i)$.
If the equality holds, we call such an expression minimal.
We study the properties of the singular degree and of minimal rational expressions.
These results are important for the computations 
involved in the Lenard-Magri scheme of integrability.
\end{abstract}

\section{Introduction}\label{sec:1}
\label{sec:intro}

Let $\mc K$ be a field with a derivation $\partial$ (this is called a differential field),
and let $\mc K[\partial]$ be the algebra of differential operators over $\mc K$
(with multiplication defined by the relation $\partial\circ f=\partial(f)+f\partial$).
The algebra $\mc K[\partial]$ embeds in the skewfield of pseudodifferential operators
$\mc K((\partial^{-1}))$
(with multiplication defined by the relation 
$\partial^m\circ f=\sum_{n=0}^\infty \binom{m}{n}\partial^n(f)\,\partial^{m-n}$, $m\in\mb Z$).
Denote by $\mc K(\partial)$ the subskewfield of $\mc K((\partial^{-1}))$
generated by $\mc K[\partial]$.
Elements of $\mc K(\partial)$ are called rational pseudodifferential operators.

In the present paper we continue the study of the algebra
$\Mat_{\ell\times\ell}\mc K(\partial)$
of $\ell\times\ell$ \emph{rational matrix pseudodifferential operators}
that we began in \cite{CDSK12,CDSK13a,CDSK13b}.

The first important property of the algebra $\Mat_{\ell\times\ell}\mc K[\partial]$
of matrix \emph{differential} operators,
is to be a left and right principal ideal ring,
hence one can talk about such arithmetic notions for this ring
as the left and right greatest common divisor 
and the left and right least common multiple of a collection of elements.
Using this one can deduce that 
a rational matrix pseudodifferential operator $H$
has a presentation in minimal terms,
very much like rational functions in one indeterminate over a field.
Namely, $H=AB^{-1}$,
where $A,B\in\Mat_{\ell\times\ell}\mc K[\partial]$,
$B$ is non-degenerate, i.e. invertible in $\Mat_{\ell\times\ell}\mc K(\partial)$,
and $A$ and $B$ are right coprime.
Moreover, for any other (right) fractional decomposition $H=\tilde{A}\tilde{B}^{-1}$
one has $\tilde{A}=AD$, $\tilde{B}=BD$,
where $D\in\Mat_{\ell\times\ell}\mc K[\partial]$ is non degenerate,
see \cite{CDSK13a,CDSK13b}.
In these papers 
we establish several equivalent properties of a minimal fractional decomposition $H=AB^{-1}$.
The most important for the present paper is that $\deg(B)$
(i.e. the degree of the Dieudonn\'e determinant of $B$)
is minimal among all (right) fractional decomposition of $H$.

We call $\deg(B)$ the \emph{singular degree} of the rational matrix pseudodifferential operator $H$.
It is a nonnegative integer, denoted by $\sdeg(H)$,
which is a ``non-commutative analogue'' of the number of poles (counting multiplicities)
of a rational function in one indeterminate.
We study the properties of the singular degree in some detail in Section \ref{sec:3.3a}.

It is not difficult to show (see Lemma \ref{20130705:lem2} below)
that for a collection $B_1,\dots,B_N$ of non-degenerate
$\ell\times\ell$ matrix differential operators one has
\begin{equation}\label{eq:1.1}
\deg\big(\text{l.c.m.}(B_1,\dots,B_N)\big)
\leq \deg(B_1)+\dots+\deg(B_N)
\,,
\end{equation}
where $\text{l.c.m.}$ denotes the left (resp. right) least common multiple.
These matrix differential operators
are called \emph{strongly left} (resp. \emph{right}) \emph{coprime}
if equality holds in \eqref{eq:1.1}.
This property implies pairwise coprimeness (see Proposition \ref{20130705:lem1}),
but it is stronger for $N\geq3$ (see Remark \ref{20130715:rem1}).
The main theorem on strong coprimeness says that
for strongly left coprime non-degenerate matrix differential
operators $B_1,\dots,B_N\in\Mat_{\ell\times\ell}\mc K[\partial]$
and vectors $F_1,\dots,F_N\in\mc K^\ell$
solving the equations $B_1F_1=\dots=B_NF_N$
there exists $F\in\mc K^\ell$ such that $F_i=C_iF$, $i=1,\dots,N$,
where $B_1C_1=\dots=B_NC_N$ is the right l.c.m. of $B_1,\dots,B_N$
(see Theorem \ref{20130705:thm}).
This result (which was proved for $N=2$ in \cite{CDSK13a})
plays an important role in our theory of minimal rational expressions.

A rational matrix pseudodifferential operator usually
comes in the form of a rational expression
\begin{equation}\label{20130611:eq1-intro}
H=\sum_{\alpha\in\mc A}
A^{\alpha}_1(B^\alpha_1)^{-1}\dots
A^{\alpha}_n(B^\alpha_n)^{-1}
\,,
\end{equation}
where 
$A^\alpha_i,B^\alpha_i\in\Mat_{\ell\times\ell}\mc K[\partial]$, $i\in\mc I=\{1,\dots,n\},\alpha\in\mc A$,
($\mc A$ is a finite index set),
and the $B^\alpha_i$'s are non-degenerate.
It is natural to ask what it means for such an expression to be in its ``minimal'' form,
and in the present paper we propose the following answer to this question.
We prove that, in general,
\begin{equation}\label{20130719:eq1}
\sdeg(H)\leq\sum_{i\in\mc I,\alpha\in\mc A}\deg(B^\alpha_i)
\,,
\end{equation}
(see Lemma \ref{20130703:lem1}),
and we say that the rational expression \eqref{20130611:eq1-intro} is \emph{minimal}
if equality holds in \eqref{20130719:eq1}.

In general, it is not easy to compute the singular degree 
of a rational expression \eqref{20130611:eq1-intro}.
One of our main results is Theorem \ref{20130801:thm},
which gives a better upper bound than \eqref{20130719:eq1},
and a lower bound, for the singular degree of $H$.
These upper and lower bounds become equal, 
thus giving an effective formula for $\sdeg(H)$,
if either the space $\mc E$ in \eqref{20130801:eq2},
or the space $\mc E^*$ in \eqref{20130801:eq3}, is zero.
As a consequence of these results,
we get, in Corollary \ref{20130801:cor},
an effective way to check when a rational expression \eqref{20130611:eq1-intro}
is minimal:
this happens if and only if both spaces $\mc E$ and $\mc E^*$ are zero.

One of the main goals of the paper is to demonstrate that this definition of minimality is the
right generalization of the minimality of a fractional decomposition.
A rational matrix pseudodifferential operator $H\in\Mat_{\ell\times\ell}\mc K(\partial)$
does not define a function $\mc K^\ell\ni\xi\mapsto P=H(\xi)\in\mc K^\ell$.
It is natural instead to define the following \emph{association relation}:
if $H$ has a rational expression as in \eqref{20130611:eq1-intro},
we reinterpret the equation ``$P=H(\xi)$''
via the association relation $\xi\ass{H}P$,
meaning that there exist 
$F^\alpha_i\in\mc K^\ell$ ($i=1,\dots,n,\,\alpha\in\mc A$) such that
\begin{equation}\label{20130611:eq2-intro}
\begin{array}{l}
\displaystyle{
\vphantom{\Big(}
\xi=B^\alpha_nF^\alpha_n
\,\,\,\,\text{ for all } \alpha \in\mc A
\,,} \\
\displaystyle{
\vphantom{\Big(}
A^\alpha_iF^\alpha_i
=B^\alpha_{i-1}F^\alpha_{i-1}
\,\,\,\,\text{ for all } 1\neq i\in\mc I\,,\alpha \in\mc A
\,,} \\
\displaystyle{
\vphantom{\Big(}
\sum_{\alpha\in\mc A} A^\alpha_1F^\alpha_1=P
\,.}
\end{array}
\end{equation}
Such association relation is a generalization 
of the $H$-association relation introduced in \cite{DSK13}.
It plays a crucial role in the theory of Hamiltonian equations,
and it is needed to develop the Lenard-Magri scheme of integrability
for a compatible pair of non-local Poisson structures
(written in the form of a general rational expression, as in \eqref{20130611:eq1-intro}).
Theorem \ref{20130704:prop},
which is our second main result,
says, in particular, that the association relation 
$\xi\ass{H}P$
is independent of the minimal rational expression \eqref{20130611:eq1-intro} for $H$.

This paper was written while the second and the third author were visiting IHES, France,
which we thank for the hospitality.

\section{Matrix differential operators and their degree}\label{sec:2}

\subsection{Matrix differential and pseudodifferential operators
and the Dieudonn\'e determinant
}\label{sec:2.1}

Let $\mathcal{K}$ be a differential field of characteristic $0$, with a 
derivation $\partial$, and let $\mathcal{C}=\ker\partial$ be the subfield of constants.
Consider the algebra ${\mathcal{K}}[\partial]$ (over $\mathcal{C}$) of 
differential operators with coefficients in $\mc K$. 
It is a subalgebra of the skewfield ${\mathcal{K}}(({\partial}^{-1}))$ 
of pseudodifferential operators with coefficients in $\mc K$.
Given $\ell\geq1$, we consider the algebra
$\Mat{}_{\ell\times\ell}\mc K[\partial]$
of $\ell\times\ell$ matrix differential operators with coefficients in $\mc K$.
It is a subalgebra of
$\Mat{}_{\ell\times\ell}\mc K((\partial^{-1}))$,
the algebra of $\ell\times\ell$ matrix pseudodifferential operators with coefficients in $\mc K$.

By definition, the \emph{Dieudonn\'e determinant} 
of $A \in \Mat_{\ell\times\ell}{\mathcal{K}}(({\partial}^{-1}))$ has the 
form $\det(A)={\det_1(A)}{{\xi}^{\deg(A)}}$ where $\det_1(A) \in \mathcal{K}$, 
$\xi$ is an indeterminate, and $\deg(A) \in \mathbb{Z}$. 
It exists and is 
uniquely defined by the following properties (see [Die43], [Art57]) :
\begin{enumerate}[$(i)$]
\item
$\det(AB)=\det(A)\det(B)$;
\item
if $A$ is upper triangular with non-zero diagonal entries 
$A_{ii} \in {\mathcal{K}}(({\partial}^{-1}))$ of degree (or order) $\deg(A_{ii})\in\mb Z$ and 
leading coefficient $a_i \in \mathcal{K}$, then 
$$
\det\nolimits_{1}(A)= \prod_{i=1}^{n} a_i \,\in\mc K
\,\,,\,\,\,\,
\deg(A)= \sum_{i=1}^{n} \deg(A_{ii})\,\in\mb Z
\,,
$$
and $\det(A)=0$ if one of the $A_{ii}$ is $0$.
\end{enumerate}
%
%

\begin{remark}\label{rem:3.3}
Let $A \in \Mat_{\ell\times\ell}{\mathcal{K}}((\partial^{-1}))$ 
and let $A^*$ be the adjoint matrix pseudodifferential operator.
If $det(A)=0$, then $det(A^*)=0$.
If $det(A) \neq 0$, then $det(A^*)={(-1)}^{\deg(A)}det(A)$.
\end{remark}

\subsection{Degree of a non-degenerate matrix}\label{sec:2.3}

A matrix $A\in\Mat{}_{\ell\times\ell}\mc K((\partial^{-1}))$  
whose Dieudonn\'e determinant is non-zero is called 
\emph{non-degenerate}. 
In this case the integer $\deg(A)$ is well defined. 
\begin{definition}\label{def:degree}
The \emph{degree} 
of a non-degenerate matrix pseudodifferential operator 
$A\in\Mat{}_{\ell\times\ell}\mc K((\partial^{-1}))$
is the integer $\deg(A)$.
\end{definition}
By the multiplicativity of the Dieudonn\'e determinant, 
we have that $\deg(AB)=\deg(A)+\deg(B)$ if both $A$ and $B$ are non-degenerate.

\begin{proposition}[\cite{CDSK13b}]
\label{cor:3.2} 
Let $A \in \Mat_{\ell\times\ell}{\mathcal{K}}[\partial]$ 
be a non-degenerate matrix differential operator. 
Then 
\begin{enumerate}[(a)]
\item
$\deg(A) \in {\mathbb{Z}}_{+}$.
\item
$A$ is an invertible element of  
$\Mat_{\ell\times\ell}{\mathcal{K}}[\partial]$
if and only if $A$ is non-degenerate and $\deg(A)=0$.
\end{enumerate}
\end{proposition}

\subsection{Right and left least common multiple}\label{sec:2.3b}

Recall the following result.
\begin{lemma}[\cite{CDSK13b}]
\label{20130703:lem2}
Let $A,B\in \Mat_{\ell\times\ell}\mc K[\partial]$
be matrix differential operators,
and assume that $B$ is non-degenerate.
\begin{enumerate}[(a)]
\item
There exist a \emph{right least common multiple}
$$
\text{right l.c.m.}(A,B)=A\widetilde{B}=B\widetilde{A}\,,
$$
with $\widetilde{A},\widetilde{B}\in\Mat_{\ell\times\ell}\mc K[\partial]$,
and $\widetilde{B}$ non-degenerate,
such that $\widetilde{A}$ and $\widetilde{B}$ are right coprime.
We have $\deg(\widetilde{B})\leq\deg(B)$,
and equality holds if and only if $A$ and $B$ are left coprime.
\item
There exist a \emph{left least common multiple}
$$
\text{left l.c.m.}(A,B)=\widetilde{B}A=\widetilde{A}B\,,
$$
with $\widetilde{A},\widetilde{B}\in\Mat_{\ell\times\ell}\mc K[\partial]$,
and $\widetilde{B}$ non-degenerate,
such that $\widetilde{A}$ and $\widetilde{B}$ are left coprime.
We have $\deg(\widetilde{B})\leq\deg(B)$,
and equality holds if and only if $A$ and $B$ are right coprime.
\end{enumerate}
\end{lemma}

In Section \ref{sec:4} we will need the following generalization of Lemma \ref{20130703:lem2}. 
\begin{lemma}\label{20130801:lem}
Let $A_i,B_i\in\Mat_{\ell\times\ell}\mc K[\partial]$, $i=1,\dots,n$,
where the $B_i$'s are non-degenerate.
Then there exist $X_1,\dots,X_n\in\Mat_{\ell\times\ell}\mc K[\partial]$,
with $X_n$ non-degenerate,
such that
\begin{equation}\label{20130801:eq1}
B_iX_i=A_{i+1}X_{i+1}
\text{ for all } i=1,\dots,n-1
\,.
\end{equation}
In this case, we have the following identity of rational matrix pseudodifferential operators
(see Section \ref{sec:3.1}):
\begin{equation}\label{20130801:eq1b}
A_1B_1^{-1}\dots A_nB_n^{-1}=(A_1X_1)(B_nX_n)^{-1}
\,.
\end{equation}
\end{lemma}
\begin{proof}
Let, by Lemma \ref{20130703:lem2}(a),
$$
B_1\widetilde{A}_2=A_2\widetilde{B}_1
\,,
$$
be the right least common multiple of $B_1$ and $A_2$,
with $\widetilde{A}_2,\widetilde{B}_1\in\Mat_{\ell\times\ell}\mc K[\partial]$
and $\widetilde{B}_1$ non-degenerate.
Let
$$
B_2\widetilde{B}_1\widetilde{A}_2=A_3\widetilde{B}_2
\,,
$$
be the right least common multiple of $B_2\widetilde{B}_1$ and $A_3$,
where $\widetilde{A}_3,\widetilde{B}_3$ lie in $\Mat_{\ell\times\ell}\mc K[\partial]$
and $\widetilde{B}_3$ is non-degenerate.
After repeating the same procedure several times, we finally let
$$
B_{n-1}\widetilde{B}_{n-2}\widetilde{A}_n=A_n\widetilde{B}_{n-1}
\,,
$$
be the right least common multiple of $B_{n-1}\widetilde{B}_{n-2}$ and $A_n$,
with $\widetilde{A}_n,\widetilde{B}_{n-1}\in\Mat_{\ell\times\ell}\mc K[\partial]$
and $\widetilde{B}_{n-1}$ non-degenerate.
Equation \eqref{20130801:eq1} then holds letting
$$
\begin{array}{l}
\displaystyle{
\vphantom{\Big(}
X_1=\widetilde{A}_2\widetilde{A}_3\dots\widetilde{A}_{n-1}\widetilde{A}_n
\,,\,\,
X_2=\widetilde{B}_1\widetilde{A}_3\dots\widetilde{A}_{n-1}\widetilde{A}_n
\,,\dots\,,} \\
\displaystyle{
\vphantom{\Big(}
X_{n-2}=\widetilde{B}_{n-3}\widetilde{A}_{n-1}\widetilde{A}_n
\,,\,\,
X_{n-1}=\widetilde{B}_{n-2}\widetilde{A}_n
\,\text{ and }\,
X_n=\widetilde{B}_{n-1}
\,.}
\end{array}
$$
The last claim is immediate since, by \eqref{20130801:eq1},
we have
$$
\begin{array}{l}
\displaystyle{
\vphantom{\Big(}
B_{n-1}^{-1}A_n=X_{n-1}X_n^{-1}
\,,\,\,
B_{n-2}^{-1}A_{n-1}X_{n-1}=X_{n-2}
\,,} \\
\displaystyle{
\vphantom{\Big(}
B_{n-3}^{-1}A_{n-2}X_{n-2}=X_{n-3}
\,\dots\,
B_1^{-1}A_2X_2=X_1
\,.}
\end{array}
$$
\end{proof}

Given an arbitrary finite number of non-degenerate matrix differential operators 
$B^1,\dots,B^N\in \Mat_{\ell\times\ell}\mc K[\partial]$,
we can consider their right (resp. left) least common multiple
$$
\begin{array}{l}
\displaystyle{
\vphantom{\Big(}
\text{right l.c.m.}
(B^1,\dots,B^N)=
B^1C^1=\dots=B^NC^N
\,,} \\
\displaystyle{
\vphantom{\Big(}
\quad \Big(\text{ resp. }\,\,
\text{left l.c.m.}
(B^1,\dots,B^N)=
C^1B^1=\dots=C^NB^N
\,\Big)
\,.}
\end{array}
$$
It can be defined as the generator of the intersection of the right (resp. left)
principal ideals in $\Mat_{\ell\times\ell}\mc K[\partial]$ generated by $B^1,\dots,B^N$.
Equivalently, it is given inductively by
(here l.c.m. means right (resp. left) l.c.m.):
\begin{equation}\label{20130705:eq5}
\text{l.c.m.}
(B^1,\dots,B^N)=
\text{l.c.m.}
\big(\text{l.c.m.}(B^1,\dots,B^{N-1}),B^N\big)
\,.
\end{equation}
\begin{lemma}\label{20130705:lem2}
Let $B^1,\dots,B^N\in \Mat_{\ell\times\ell}\mc K[\partial]$ 
be non-degenerate matrix differential operators.
We have
$$
\deg\big(\emph{right (resp. left) l.c.m.}(B^1,\dots,B^N)\big)
\leq
\deg(B^1)+\dots+\deg(B^N)
\,.
$$
\end{lemma}
\begin{proof}
For $N=2$ the claim is Lemma \ref{20130703:lem2}.
For arbitrary $N\geq2$, it follows inductively by equation \eqref{20130705:eq5}.
\end{proof}

\subsection{Strongly coprime matrices}\label{sec:2.3c}

\begin{definition}\label{20130705:def}
We say that the non-degenerate matrix differential operators
$B^1,\dots,B^N\in \Mat_{\ell\times\ell}\mc K[\partial]$ 
are \emph{strongly left (resp. right) coprime} if
$$
\deg\big(\text{right (resp. left) l.c.m.}(B^1,\dots,B^N)\big)
=
\deg(B^1)+\dots+\deg(B^N)
\,.
$$
\end{definition}
Note that, by Lemma \ref{20130703:lem2},
strong coprimeness is equivalent to coprimeness if $N=2$.
\begin{proposition}\label{20130705:lem1}
Let $B^1,\dots,B^N\in \Mat_{\ell\times\ell}\mc K[\partial]$
be non-degenerate and strongly left (resp. right) coprime.
Then they are pairwise left (resp. right) coprime.
\end{proposition}
\begin{proof}
Let $B$ be the right (resp. left) least common multiple of $B^1,\dots,B^N$.
By assumption, $\deg(B)=\deg(B^1)+\dots+\deg(B^N)$.
Let now $\widetilde{B}^{N-1}$ be the right (resp. left) least common multiple 
of $B^{N-1}$ and $B^N$.
By the inductive formula \eqref{20130705:eq5} we have
$$
B=\text{right (resp. left) l.c.m} (B^1,\dots,B^{N-2},\widetilde{B}^{N-1})\,,
$$
and therefore, by Lemma \ref{20130705:lem2}, we have
\begin{equation}\label{20130705:eq6}
\begin{array}{l}
\displaystyle{
\vphantom{\Big(}
\deg(B)\leq\deg(B^1)+\dots+\deg(B^{N-2})+\deg(\widetilde{B}^{N-1})
} \\
\displaystyle{
\vphantom{\Big(}
\leq\deg(B^1)+\dots+\deg(B^{N-2})+\deg(B^{N-1})+\deg(B^N)
\,.}
\end{array}
\end{equation}
It follows that all inequalities in \eqref{20130705:eq6} are actually equalities,
and therefore, in particular, 
$$
\deg(\widetilde{B}^{N-1})=\deg(B^{N-1})+\deg(B^N)\,.
$$
By Lemma \ref{20130703:lem2}
this is equivalent to say that  $B^{N-1}$ and $B^N$ are left (resp. right) coprime.
The same argument works for any other pair $(B^i,B^j)$.
\end{proof}
\begin{remark}\label{20130715:rem1}
Strong coprimeness is stronger than pairwise coprimeness
of $N\geq3$ differential operators.
To see this, consider the differential operators
$\partial,\partial+\frac1x,\partial+\frac1{x+1}$ with coefficients 
in the field $\mb F(x)$ of rational functions in $x$.
They are obviously pairwise left coprime.
On the other hand, their right least common multiple is
$$
\partial^2
=\Big(\partial+\frac1x\Big)\circ\Big(\partial-\frac1x\Big)
=\Big(\partial+\frac1{x+1}\Big)\circ\Big(\partial-\frac1{x+1}\Big)
\,,
$$
which has degree $2<1+1+1$.
Hence, $\partial,\partial+\frac1x$ and $\partial+\frac1{x+1}$
are not strongly coprime.
\end{remark}

Recall the following result:
\begin{theorem}[\cite{CDSK13a}]\label{CR13}
Let $A,B\in\Mat_{\ell\times\ell}\mc K[\partial]$ be left coprime matrix differential operators,
with $B$ non-degenerate.
Let $A\widetilde{B}=B\widetilde{A}$ be their right least common multiple.
Then, for every $X,Y\in\mc K^\ell$ solving the equation $AX=BY$,
there exists $Z\in\mc K^\ell$ such that $X=\widetilde{B}Z$ and $Y=\widetilde{A}Z$.
\end{theorem}
We can generalize this to an arbitrary number of strongly coprime operators.
\begin{theorem}\label{20130705:thm}
Let $B^1,\dots,B^N\in\Mat_{\ell\times\ell}\mc K[\partial]$ 
be strongly left coprime non-degenerate matrix differential operators.
Let $B=B^1C^1=\dots=B^NC^N$ be their right least common multiple.
Then, for every $F^1,\dots,F^N\in\mc K^\ell$ solving 
the equations 
\begin{equation}\label{20130705:eq7}
B^1F^1=\dots=B^NF^N\,,
\end{equation}
there exists $F\in\mc K^\ell$ such that $F^\alpha=C^\alpha F$ for every $\alpha=1,\dots,N$.
\end{theorem}
\begin{proof}
For $N=2$ the claim holds by Theorem \ref{CR13}.
For $N\geq3$, we prove the claim by induction on $N$.
Let
$$
\widetilde{B}^2=\text{right l.c.m.}(B^2,\dots,B^N)=B^2D^2=\dots=B^ND^N
\,.
$$
By the strong coprimeness of $B^1,\dots,B^N$ and Lemma \ref{20130705:lem2},
we immediately have that
$\deg(\widetilde{B}^2)=\deg(B^2)+\dots+\deg(B^N)$,
and that $B^1$ and $\widetilde{B}^2$ are left coprime.
Since $B^2F^2=\dots=B^NF^N$,
by the inductive assumption there exists 
$\widetilde{F}^2\in\mc K^\ell$ such that 
\begin{equation}\label{20130705:eq9}
F^2=D^2 \widetilde{F}^2,\dots,F^N=D^N \widetilde{F}^2
\,.
\end{equation}
Hence, by the first equation in \eqref{20130705:eq7}, we have
\begin{equation}\label{20130705:eq8}
B^1 F^1=\widetilde{B}^2\widetilde{F}^2
\,.
\end{equation}
On the other hand, 
by the inductive formula \eqref{20130705:eq5}
we have
$$
B
=\text{right l.c.m.}(B^1,\widetilde{B}^2)
=B^1C^1=\widetilde{B}^2E
\,,
$$
and, therefore,
\begin{equation}\label{20130705:eq10}
C^2=D^2E,\dots,C^N=D^NE
\,.
\end{equation}
Since $B^1$ and $\widetilde{B}^2$ are left coprime,
by equation \eqref{20130705:eq8} and Theorem \ref{CR13}
there exists $F\in\mc K^{\ell}$ such that
$$
F^1=C^1F
\,\text{  and }\,
\widetilde{F}^2=EF\,.
$$
These equations, combined with \eqref{20130705:eq9} and \eqref{20130705:eq10},
prove the claim.
\end{proof}
\begin{remark}\label{20130725:rem}
The example in Remark \ref{20130715:rem1} shows that Theorem \ref{20130705:thm}
may fail for pairwise left coprime $B_i$'s.
Indeed, let, as in Remark \ref{20130715:rem1}, 
$B_1=\partial,\,B_2=\partial+\frac1x,\,B_3=\partial+\frac1{x+1}$, 
and $C_1=\partial,\,C_2=\partial-\frac1x,\,C_3=\partial-\frac1{x+1}$,
so that $B_1C_1=B_2C_2=B_3C_3$ is the right l.c.m. of $B_1$, $B_2$, $B_3$.
Let also $F_1=1$, $F_2=\frac1x$ and $F_3=\frac\alpha{x+1}$,
where $\alpha$ is a constant.
They solve the equations $B_1F_1=B_2F_2=B_3F_3=0$.
On the other hand, the only function $F$ solving $C_1F=F_1$ and $C_2F=F_2$
is $F=x-1$. Such $F$ solves also the equation $C_3F=F_3$ if and only if $\alpha=2$.
\end{remark}

\subsection{Linearly closed differential fields}\label{sec:2.4}

A differential field $\mc K$ is called \emph{linearly closed}
if every homogeneous linear differential equation of order $n\geq1$,
\begin{equation}\label{20120121:eq1}
a_nu^{(n)}+\dots+a_1u'+a_0u=0\,,
\end{equation}
with $a_0,\dots,a_n$ in $\mc K$, $a_n\neq0$,
has a non-zero solution $u\in\mc K$.

It is easy to show that
the solutions of equation \eqref{20120121:eq1} in a differential field $\mc K$
form a vector space over the field of constants $\mc C$
of dimension less than or equal to $n$,
and equal to $n$ if $\mc K$ is linearly closed (see e.g. \cite{DSK11}).
\begin{proposition}[\cite{CDSK13b}]
\label{CDSK13b:prop1}
If $A\in\Mat_{\ell\times\ell}{\mathcal{K}}[\partial]$
is a non-degenerate matrix differential operator 
and $b\in\mc K^\ell$,
then the inhomogeneous system of linear differential equations
in $u=\big(u_i\big)_{i=1}^n$,
\begin{equation}\label{20120121:eq5}
A(\partial)u=b\,,
\end{equation}
admits the affine space (over $\mc C$) of solutions
of dimension less than or equal to $\deg (A)$,
and equal to $\deg(A)$ if $\mc K$ is linearly closed. 
\end{proposition}
\begin{defprop}[\cite{Mag94} (see also \cite{CDSK13b})]
\label{CDSK13b:prop2}
Let $\mc K$ be a differential field with subfield of constants $\mc C$,
and let $\bar{\mc C}$ be the algebraic closure of $\mc C$.
Then there exists a unique (up to isomorphism)
minimal linearly closed extension $\mc K \subset \mc L$
with subfield of constants $\bar{\mc C}$,
called the \emph{linear closure} of $\mc K$.
\end{defprop}

\begin{corollary}\label{DSK11:prop}
Let $\mc K$ be a differential field with subfield of constants $\mc C$.
Let $\bar{\mc C}$ be the algebraic closure of $\mc C$,
and let $\mc L$ be the linear closure of $\mc K$.
Let $A\in\Mat_{\ell\times\ell}\mc K[\partial]$
be a non-degenerate matrix differential operator.
Then,
$$
\deg(A)=\dim_{\bar{\mc C}}\ker{}_{\mc L}(A)\,,
$$
where $\ker_{\mc L}(A)$ denotes the kernel of $A$,
considered as a map $\mc L^{\ell}\to\mc L^{\ell}$.
\end{corollary}

\section{Singular degree of a rational matrix pseudodifferential operator}\label{sec:3}

\subsection{Rational matrix pseudodifferential operators}\label{sec:3.1}

Throughout the rest of the paper
we let $\mc K$ be a differential field 
with derivation $\partial$
and with subfield of constants $\mc C$,
we let $\bar{\mc C}$ be the algebraic closure of $\mc C$
and $\mc L$ be the linear closure of $\mc K$.

The algebra $\mc K(\partial)$ of \emph{rational pseudodifferential operators} over $\mc K$ is,
by definition, the smallest subskewfield of $\mc K((\partial^{-1}))$ containing $\mc K[\partial]$.
Any rational pseudodifferential operator $L\in\mc K(\partial)$
admits a fractional decomposition $h=ab^{-1}$,
with $a,b\in\mc K[\partial]$ (see e.g. \cite{CDSK12}).

A matrix $H \in\Mat_{\ell\times\ell}({\mathcal{K}}(\partial))$ is called a 
\emph{rational matrix pseudodifferential operator}. 
In other words, all the 
entries of such a matrix have the form $h_{ij}={a_{ij}}{b_{ij}}^{-1}$, 
$i,j=1,\dots,\ell$, where $a_{ij}, b_{ij} \in {\mathcal{K}}[\partial]$ and all 
${b_{ij}} \neq 0$. 
Denoting by $b$ a right common multiple of the $b_{ij}$'s
(see e.g. \cite{CDSK12}),
we see that $H$ 
admits a fractional decomposition $H=AB^{-1}$,
where $A,B\in\Mat_{\ell\times\ell}\mc K[\partial]$
and $B=b\id$ is non-degenerate.

\subsection{Minimal fractional decomposition for a rational matrix pseudodifferential operator
and singular degree}
\label{sec:3.1a}

\begin{definition}
A right fractional decomposition $H=AB^{-1}$, where $A,B \in 
M_{\ell\times\ell}{\mathcal{K}}[\partial]$ and $B$ non-degenerate, 
is called \emph{minimal} if $\deg(B)$ ( $\in {\mathbb{Z}}_{+}$) is minimal among all 
possible right fractional decompositions of $H$.
\end{definition}


\begin{theorem}[\cite{CDSK13b}]
\label{th:6.4}
\begin{enumerate}[(a)]
\item
Let $H\in \Mat_{\ell\times\ell}\mc K(\partial)$,
and let $H=AB^{-1}$ be a right fractional decomposition for $H$,
with $A,B\in\Mat_{\ell\times\ell}\mc K[\partial]$ and $B$ non-degenerate.
The following conditions are equivalent:
\begin{enumerate}[$(i)$]
\item
$H=AB^{-1}$ \emph{minimal};
\item
$A$ and $B$ are \emph{right coprime}, i.e.
if $A=A_1D$ and $B=B_1D$, with $A_1,B_1,D\in M_n({\mathcal{K}}[\partial])$,
then $D$ is invertible in $M_{\ell\times\ell}({\mathcal{K}}[\partial])$;
\item
$CA+DB=\id$ for some $C,D\in\Mat_{\ell\times\ell}\mc K[\partial]$
\emph{(Bezout identity)};
\item
$\ker_{\mc L}A\cap\ker_{\mc L}B=0$.
\end{enumerate}
\item
If ${A_0}{B_0}^{-1}$ is a minimal fractional decomposition of the fraction 
$H=AB^{-1}$, then one can find a non-degenerate matrix differential operator $D$ such that 
$A={A_0}D$ and $B={B_0}D$.
\item
A minimal right fractional decomposition $H=AB^{-1}
\in\Mat_{\ell\times\ell}\mc K(\partial)$,
and a minimal left fractional decomposition $H=B_1^{-1}A_1$
(i.e. with $B_1\in\Mat_{\ell\times\ell}\mc K[\partial]$
non-degenerate of minimal possible degree),
have denominators of the same degree: $\deg(B)=\deg(B_1)$.
\end{enumerate}
\end{theorem}


\begin{definition}\label{20130610:def}
Let $H\in\Mat_{\ell\times\ell}\mc K(\partial)$
be a rational matrix pseudodifferential operator,
and let $H=AB^{-1}$ be its minimal fractional decomposition,
with $A,B\in\Mat_{\ell\times\ell}\mc K[\partial]$ and $B$ non-degenerate.
The \emph{singular degree} of $H$ is the non-negative integer
$\sdeg(H)=\deg(B)$.
\end{definition}

\subsection{Some properties of the singular degree}\label{sec:3.3a}

\begin{proposition}\label{20130610:lem1}
Let $H\in\Mat_{\ell\times\ell}\mc K(\partial)$,
and let $D\in\Mat_{\ell\times\ell}\mc K[\partial]$ be a non-degenerate matrix
such that $HD\in\Mat_{\ell\times\ell}\mc K[\partial]$.
Then
$$
\sdeg(H)=\dim_{\bar{\mc C}}\big((HD)(\ker{}_{\mc L}D)\big)
\,.
$$
\end{proposition}
\begin{proof}
By assumption, 
$D\in\Mat_{\ell\times\ell}\mc K[\partial]$ is a non-degenerate matrix
such that $C=HD\in\Mat_{\ell\times\ell}\mc K[\partial]$,
hence $H=CD^{-1}$.
Let $H=AB^{-1}$,
with $A,B\in\Mat_{\ell\times\ell}\mc K[\partial]$ and $B$ non-degenerate,
be a minimal fractional decomposition for $H$.
Then, by Theorem \ref{th:6.4}(b),
there exists a non-degenerate matrix $E\in\Mat_{\ell\times\ell}\mc K[\partial]$
such that $C=AE$ and $D=BE$.
We claim that
\begin{equation}\label{20130702:eq1}
C(\ker{}_{\mc L}D)=A(\ker{}_{\mc L}B)\,.
\end{equation}
Indeed, let $y\in C(\ker_{\mc L}D)$.
Namely, $y=C(k)\in\mc L^{\ell}$, with $k\in\ker_{\mc L}D$.
Then, $E(k)\in\ker B$, and $y=C(k)=AE(k)=A(Ek)\in A(\ker_{\mc L}B)$,
proving the inclusion $\subset$.
For the opposite inclusion,
let $x\in A(\ker_{\mc L}B)$.
Namely, $x=A(h)\in\mc L^{\ell}$, with $h\in\ker_{\mc L}B$.
Since $\mc L$ is a linearly closed differential field
and $E$ is non-degenerate, by Proposition \ref{CDSK13b:prop1}
there exists $k\in\mc L^\ell$ such that $h=E(k)$.
Therefore, $D(k)=BE(k)=B(h)=0$,
and $C(k)=AE(k)=A(h)=x$,
so that $x\in C(\ker_{\mc L}D)$.

%
By Definition \ref{20130610:def}, we have $\sdeg(H)=\deg(B)$.
By Corollary \ref{DSK11:prop},
we have $\deg(B)=\dim_{\bar{\mc C}}(\ker_{\mc L}B)$.
On the other hand,
since, by assumption, $H=AB^{-1}$ is a minimal fractional decomposition,
by Theorem \ref{th:6.4}(a)(iv) we have
$\ker_{\mc L}A\cap\ker_{\mc L}B=0$,
and therefore $\dim_{\bar{\mc C}}(\ker_{\mc L}B)=\dim_{\bar{\mc C}}A(\ker_{\mc L}B)$.
The claim follows by the above observations and equation \eqref{20130702:eq1}.
\end{proof}

\begin{proposition}\label{20130610:lem2}
For $H\in\Mat_{\ell\times\ell}\mc K(\partial)$,
we have $\sdeg(H)=0$ if and only if $H\in\Mat_{\ell\times\ell}\mc K[\partial]$.
\end{proposition}
\begin{proof}
The if part is obvious, by Definition \ref{20130610:def}.
The only if part follows from Proposition \ref{cor:3.2}(b).
\end{proof}

\begin{proposition}\label{20130725:lem1}
If $A\in\Mat_{\ell\times\ell}\mc K[\partial]$ and $H\in\Mat_{\ell\times\ell}\mc K(\partial)$,
then
$\sdeg(A+H)=\sdeg(H)$.
\end{proposition}
\begin{proof}
If $H=A_1B_1^{-1}$ is a minimal fractional decomposition for $H$,
then, clearly, $A+H=(AB_1+A_1)B_1^{-1}$ is a minimal fractional decomposition for $A+H$.
The claim follows.
\end{proof}

\begin{proposition}\label{20130702:lem1}
For $H\in\Mat_{\ell\times\ell}\mc K(\partial)$,
we have $\sdeg(H)=\sdeg(H^*)$.
\end{proposition}
\begin{proof}
Clearly, $H=AB^{-1}$ is a minimal right fractional decomposition for $H$,
if and only if $H^*={B^*}^{-1}A^*$ is a minimal left fractional decomposition for $H^*$.
Therefore, by Theorem \ref{th:6.4}(c) and Remark \ref{rem:3.3},
we obtain that $\sdeg(H^*)=\deg(B^*)=\deg(B)$.
\end{proof}

\begin{proposition}\label{20130725:lem2}
Let $p_1,\dots,p_s$ be positive integers such that $p_1+\dots+p_s=\ell$,
and let $H=\big(H_{ij}\big)_{i,j=1}^s$ be a block form for 
the rational $\ell\times\ell$ matrix pseudodifferential operator $H$,
where $H_{ij}\in\Mat_{p_i\times p_j}\mc K(\partial)$ for every $i,j=1,\dots,s$.
Assume, moreover, that $H_{ij}\in\Mat_{p_1\times p_j}\mc K[\partial]$ if $i\neq j$.
Then
$$
\sdeg(H)=\sdeg(H_{11})+\dots+\sdeg(H_{ss})
\,.
$$
\end{proposition}
\begin{proof}
For every $i=1,\dots,s$, let $H_{ii}=A_iB_i^{-1}$ be a fractional decomposition 
for $H_{ii}\in\Mat_{p_i\times p_i}\mc K(\partial)$.
The matrix
$$
B=
\left(\begin{array}{ccc}
B_1 & & 0 \\
& \ddots & \\
0 & & B_s
\end{array}\right)
\in\Mat{}_{\ell\times\ell}\mc K[\partial]
$$
is clearly non-degenerate.
Then $HB$ lies in $\Mat_{\ell\times\ell}\mc K[\partial]$,
$\ker B=\ker(B_1)\oplus\dots\oplus\ker(B_s)$,
and $HB(\ker(B_1)\oplus\dots\oplus\ker(B_s))=A_1\ker(B_1)\oplus\dots\oplus A_s\ker(B_s)$.
The claim follows by Proposition \ref{20130610:lem1}.
\end{proof}

\begin{proposition}\label{20130610:lem3}
Let $H\in\Mat_{\ell\times\ell}\mc K(\partial)$.
\begin{enumerate}[(a)]
\item
If $H=AB^{-1}$ is a right fractional decomposition for $H$, 
with $A,B\in\Mat_{\ell\times\ell}\mc K[\partial]$ and $B$ non-degenerate,
then
\begin{equation}\label{20130702:eq2}
\sdeg(H)=\deg(B)-\dim_{\bar{\mc C}}(\ker{}_{\mc L}A\cap\ker{}_{\mc L}B)
\,.
\end{equation}
\item
If $H=B^{-1}A$ is a left fractional decomposition for $H$, 
with $A,B\in\Mat_{\ell\times\ell}\mc K[\partial]$ and $B$ non-degenerate,
then
\begin{equation}\label{20130702:eq3}
\sdeg(H)=\deg(B)-\dim_{\bar{\mc C}}(\ker{}_{\mc L}A^*\cap\ker{}_{\mc L}B^*)
\,.
\end{equation}
\end{enumerate}
\end{proposition}
\begin{proof}
Let $H=A_0B_0^{-1}$ be a minimal fractional decomposition for $H$,
with $A_0,B_0\in\Mat_{\ell\times\ell}\mc K[\partial]$ and $B_0$ non-degenerate.
By Theorem \ref{th:6.4}(b)
there exists a non-degenerate $E\in\Mat_{\ell\times\ell}\mc K[\partial]$
such that $A=A_0E$ and $B=B_0E$.
Since $B_0$ and $E$ are both non-degenerate, and $B=B_0E$,
we have
\begin{equation}\label{20130702:eq4}
\deg(B)=\deg(B_0)+\deg(E)\,.
\end{equation}
By assumption $H=A_0B_0^{-1}$ is a minimal fraction,
and therefore by Theorem \ref{th:6.4}(a)(iv) we have
$\ker_{\mc L}A_0\cap\ker_{\mc L}B_0=0$.
It immediately follows that
$\ker{}_{\mc L}A\cap\ker{}_{\mc L}B=\ker{}_{\mc L}E$.
Therefore, by Corollary \ref{DSK11:prop},
\begin{equation}\label{20130702:eq5}
\deg(E)=\dim_{\bar{\mc C}}\big(\ker{}_{\mc L}A\cap\ker{}_{\mc L}B\big)
\,.
\end{equation}
Equation \eqref{20130702:eq2}
follows from equations \eqref{20130702:eq4} and \eqref{20130702:eq5},
and the fact that, by Definition \ref{20130610:def}, $\sdeg(H)=\deg(B_0)$.

In order to prove part (b),
note that $H^*=A^*{B^*}^{-1}$.
Therefore, by part (a),
$\sdeg(H^*)=\deg(B^*)-\dim_{\bar{\mc C}}(\ker{}_{\mc L}A^*\cap\ker{}_{\mc L}B^*)$.
Equation \eqref{20130702:eq3}
follows from Proposition \ref{20130702:lem1} and the fact that, by Remark \ref{rem:3.3},
$\deg(B)=\deg(B^*)$.
\end{proof}

\begin{proposition}\label{20130725:lem3}
Let $H=AB^{-1}C\in\Mat_{\ell\times\ell}\mc K(\partial)$,
where $A,B,C\in\Mat_{\ell\times\ell}\mc K[\partial]$ are matrix differential operators,
and $B$ is non-degenerate.
\begin{enumerate}[(a)]
\item
If $B$ and $C$ are left coprime, 
then $\sdeg(H)=\deg(B)-\dim_{\bar{\mc C}}(\ker_{\mc L} A\cap\ker_{\mc L} B)$.
\item
If $A$ and $B$ are right coprime, 
then $\sdeg(H)=\deg(B)-\dim_{\bar{\mc C}}(\ker_{\mc L} B^*\cap\ker_{\mc L} C^*)$.
\item
If $A$ and $B$ are right coprime and $B$ and $C$ are left coprime,
then $H=AB^{-1}C$ is a minimal rational expression for $H$, i.e. $\sdeg(H)=\deg(B)$.
\end{enumerate}
\end{proposition}
\begin{proof}
We start proving claim (c) (which is a special case of (a) and (b)).
Let $BC_1=CB_1$ be the right l.c.m. of $B$ and $C$.
In particular,
$B_1$ and $C_1$ are right coprime.
Moreover, since, by assumption, $B$ and $C$ are left coprime,
we have by Lemma \ref{20130703:lem2}(a)
that $\deg(B_1)=\deg(B)$.
We then have $H=(AC_1)B_1^{-1}$,
and we claim that this is a minimal fractional decomposition for $H$
(so that $\sdeg(H)=\deg(B_1)=\deg(B)$.)
Tho do so, it suffices to prove, by Theorem \ref{th:6.4}(a),
that $\ker_{\mc L}(AC_1)\cap\ker_{\mc L}(B_1)=0$.
Indeed, let $F\in\ker_{\mc L}(AC_1)\cap\ker_{\mc L}(B_1)$.
We have 
\begin{equation}\label{20130725:eq1}
AC_1F=0
\,\text{ and }\,
B_1F=0
\,.
\end{equation}
Applying $C$ to the second equation, we get 
\begin{equation}\label{20130725:eq2}
BC_1F=CB_1F=0\,.
\end{equation}
Combining the first equation in \eqref{20130725:eq1} and equation \eqref{20130725:eq2},
we get that $C_1F\in\ker_{\mc L}A\cap\ker_{\mc L}B=0$, since, by assumption,
$A$ and $B$ are right coprime.
But then, by the second equation in \eqref{20130725:eq1} we get that
$F\in\ker_{\mc L}B_1\cap\ker_{\mc L}C_1=0$,
since $B_1$ and $C_1$ are right coprime as well.
This completes the proof of part (c).

Next, we prove part (a).
Let $D\in\Mat_{\ell\times\ell}\mc K[\partial]$ be the right greatest common divisor of $A$ and $B$.
In other words, $D$ is non-degenerate, $A=A_0D$, $B=B_0D$,
and $A_0$ and $B_0$ are right coprime.
It is immediate to cheek that
$\ker_{\mc L}D=\ker_{\mc L}A\cap\ker_{\mc L}B$.
Hence, by Corollary \ref{DSK11:prop}, we have
\begin{equation}\label{20130725:eq3}
\deg(D)=\dim_{\bar{\mc C}}(\ker{}_{\mc L}A\cap\ker{}_{\mc L}B)
\,.
\end{equation}
Since, by assumption, $B$ and $C$ are left coprime,
we have, a fortiori, that $B_0$ and $C$ are left coprime as well.
Hence, the expression $H=A_0B_0^{-1}C$ satisfies all the assumptions of part (c),
and we conclude that $\sdeg(H)=\deg(B_0)$.
Claim (a) follows from equation \eqref{20130725:eq3}
and the fact that $\deg(B)=\deg(B_0)+\deg(D)$.

Finally, part (b) follows from part (a) and Proposition \ref{20130702:lem1}.
\end{proof}

\begin{proposition}\label{20130703:lem3}
For $H,K\in\Mat_{\ell\times\ell}\mc K(\partial)$, we have
\begin{enumerate}[(a)]
\item
$\sdeg(HK)\leq\sdeg(H)+\sdeg(K)$;
\item
$\sdeg(H+K)\leq\sdeg(H)+\sdeg(K)$.
\end{enumerate}
\end{proposition}
\begin{proof}
Let $H=AB^{-1}$ and $K=CD^{-1}$ be minimal fractional decompositions 
for $H$ and $K$ respectively,
so that, by definition, $\sdeg(H)=\deg(B)$ and $\sdeg(K)=\deg(D)$.
By Lemma \ref{20130703:lem2}(a),
there exist right corpime matrices 
$\widetilde{B},\widetilde{C}\in\Mat_{\ell\times\ell}\mc K[\partial]$
such that
$\widetilde{B}$ is non-degenerate with $\deg(\widetilde{B})\leq\deg(B)$,
and right l.c.m.$(B,C)=B\widetilde{C}=C\widetilde{B}$.
Hence, $HK=AB^{-1}CD^{-1}=A\widetilde{C}\big(D\widetilde{B}\big)^{-1}$,
and therefore, by the definition of the singular degree,
$$
\begin{array}{l}
\displaystyle{
\vphantom{\Big(}
\sdeg(HK)\leq\deg(D\widetilde{B})
=\deg(D)+\deg(\widetilde{B})
} \\
\displaystyle{
\vphantom{\Big(}
\leq\deg(D)+\deg(B)
=\sdeg(H)+\sdeg(K)
\,.}
\end{array}
$$

Similarly, by Lemma \ref{20130703:lem2}(b),
there exist left corpime matrices 
$\widetilde{B}_1,\widetilde{D}\in\Mat_{\ell\times\ell}\mc K[\partial]$
such that
$\widetilde{B}_1$ is non-degenerate with $\deg(\widetilde{B}_1)\leq\deg(B)$,
and left l.c.m.$(B,D)=\widetilde{B}_1D=\widetilde{D}B$.
Hence, 
$$
H+K
=A\widetilde{D}(B\widetilde{D})^{-1}+C\widetilde{B}_1(D\widetilde{B}_1)^{-1}
=\big(A\widetilde{D}+C\widetilde{B}_1\big)(D\widetilde{B}_1)^{-1}
\,.
$$
Therefore,
$$
\begin{array}{l}
\displaystyle{
\vphantom{\Big(}
\sdeg(H+K)\leq\deg(D\widetilde{B}_1)
=\deg(D)+\deg(\widetilde{B}_1)
} \\
\displaystyle{
\vphantom{\Big(}
\leq\deg(D)+\deg(B)
=\sdeg(H)+\sdeg(K)
\,.}
\end{array}
$$
\end{proof}

\subsection{Basic Lemma}\label{sec:3.3b}

\begin{lemma}\label{20130705:lem3}
Let $A^\alpha,B^\alpha\in\Mat_{\ell\times\ell}\mc K[\partial]$, $\alpha=1,\dots,N$,
where $B^\alpha$ is non-degenerate for every $\alpha$.
Consider the rational matrix pseudodifferential operator
\begin{equation}\label{20130705:eq11}
H=A^1(B^1)^{-1}+\dots+A^N(B^N)^{-1}\,,
\end{equation}
and assume that
\begin{equation}\label{20130705:eq12}
\sdeg(H)=\deg(B^1)+\dots+\deg(B^N)\,.
\end{equation}
(In other words, \eqref{20130705:eq11} is a minimal rational expression for $H$,
cf. Definition \ref{20130611:defa} below.)
Let
\begin{equation}\label{20130705:eq13}
B=B^1C^1=\dots=B^NC^N
\,,
\end{equation}
be the right least common multiple of $B^1,\dots,B^N$.
Then:
\begin{enumerate}[(a)]
\item
Each summand $A^\alpha(B^\alpha)^{-1}$ is a minimal fractional decomposition.
\item
The non-degenerate matrices $B^1,\dots,B^N$ are strongly left coprime 
(see Definition \ref{20130705:def}).
\item
$H=(A^1C^1+\dots+A^NC^N)B^{-1}$
is a minimal fractional decomposition for $H$.
\end{enumerate}
\end{lemma}
\begin{proof}
By equation \eqref{20130705:eq11} and Proposition \ref{20130703:lem3}(b),
we have
$$
\begin{array}{l}
\displaystyle{
\vphantom{\Big(}
\sdeg(H)\leq\sdeg(A^1(B^1)^{-1})+\dots+\sdeg(A^N(B^N)^{-1})
} \\
\displaystyle{
\vphantom{\Big(}
\leq\deg(B^1)+\dots+\deg(B^N)
\,.}
\end{array}
$$
Hence, by the assumption \eqref{20130705:eq12},
all inequalities above are in fact equalities.
In particular, 
$\sdeg(A^\alpha(B^\alpha)^{-1})=\deg(B^\alpha)$ for every $\alpha=1,\dots,N$,
proving (a).

By the obvious identity $H=(A^1C^1+\dots+A^NC^N)B^{-1}$
and Lemma \ref{20130705:lem2} we have
$$
\sdeg(H)\leq\deg(B)\leq\deg(B^1)+\dots+\deg(B^N)\,.
$$
Again, by the assumption \eqref{20130705:eq12},
all inequalities above are equalities.
In particular $\deg(B)=\deg(B^1)+\dots+\deg(B^N)$, proving (b),
and $\sdeg(H)=\deg(B)$, proving (c).
\end{proof}
\begin{remark}\label{20130715:rem2}
Clearly, conditions (a), (b) and (c) imply that \eqref{20130705:eq11} is a minimal
rational expression (i.e. \eqref{20130705:eq12} holds).
On the other hand, conditions (a) and (b) alone
are not sufficient for the minimalty of \eqref{20130705:eq11}.
To see this, consider the rational expression
\begin{equation}\label{20130715:eq1}
H=e^{-x}\partial^{-1}+1(\partial+1)^{-1}\,.
\end{equation}
Clearly, $e^{-x}\circ\partial^{-1}$ and $1\circ(\partial+1)^{-1}$ 
are minimal fractional decompositions,
and $\partial$ and $\partial+1$ are left coprime (hence strongly left coprime).
Hence, conditions (a) and (b) of Lemma \ref{20130705:lem3} hold.
On the other hand, we have
$\partial(\partial+1)=\Big(\partial+\frac1{1+e^{-x}}\Big)\Big(\partial+\frac{e^{-x}}{1+e^{-x}}\Big)$,
and $e^{-x}(\partial+1)+\partial=(1+e^{-x})\Big(\partial+\frac{e^{-x}}{1+e^{-x}}\Big)$,
so that
$$
H=(1+e^{-x})\Big(\partial+\frac1{1+e^{-x}}\Big)^{-1}\,.
$$
Hence, $\sdeg(H)=1<1+1$, and condition \eqref{20130705:eq12} fails.
\end{remark}

\section{The association relation}\label{sec:4}

\subsection{Definition of the association relation}\label{sec:4.1}

Recall the definition of $H$-association relation from \cite{DSK13}:
\begin{definition}\label{20130703:def1}
Given a rational matrix pseudodifferential operator 
$H\in\Mat_{\ell\times\ell}\mc K(\partial)$,
we say that the elements $\xi,P\in\mc K^\ell$
are $H$-\emph{associated},
and we denote this by $\xi\ass{H} P$,
if there exist
a fractional decomposition $H=AB^{-1}$,
with $A,B\in\Mat_{\ell\times\ell}\mc K[\partial]$
and $B$ non-degenerate,
and an element $F\in\mc K^\ell$,
such that $\xi=BF$ and $P=AF$.
\end{definition}
\begin{remark}
One can generalize the notion of association relation $\xi\ass{H} P$
for $\xi$ and $P$ with entries in a differential domain $\mc V$ (see e.g. \cite{DSK13}).
However the solution $F$ of the equations $\xi=BF$ and $P=AF$
is allowed to have entries in the field of fractions $\mc K$.
The same remark applies to Definition \ref{20130703:def2} below.
\end{remark}

We want to generalize the above association relation to an arbitrary 
\emph{rational expression} for $H$,
namely an expression of the form
\begin{equation}\label{20130611:eq1}
H=\sum_{\alpha\in\mc A}
A^{\alpha}_1(B^\alpha_1)^{-1}\dots
A^{\alpha}_n(B^\alpha_n)^{-1}
\,,
\end{equation}
with $A^\alpha_i,B^\alpha_i\in\Mat_{\ell\times\ell}\mc K[\partial]$
and $B^\alpha_i$ non-degenerate, for all $i\in\mc I,\alpha\in\mc A$.
(Here and further, we let $\mc I=\{1,\dots,n\}$ and $\mc A$ be a finite index set,
of cardinality $|\mc A|=N$.)

\begin{definition}\label{20130703:def2}
Given matrices
$A^\alpha_i,B^\alpha_i\in\Mat_{\ell\times\ell}\mc K[\partial]$, 
$i\in\mc I,\alpha\in\mc A$,
with $B^\alpha_i$ non-degenerate for all $i,\alpha$,
we say that the elements $\xi,P\in\mc K^\ell$
are $\{A^\alpha_i,B^\alpha_i\}_{i,\alpha}$-\emph{associated}
over the differential field extension $\mc K\subset\mc K_1$,
and we denote this by 
\begin{equation}\label{20130703:eq1}
\xi\assk{\{A^\alpha_i,B^\alpha_i\}_{i,\alpha}}{\mc K_1} P
\,,
\end{equation}
if there exist
$F^\alpha_i\in\mc K_1^\ell$, $i\in\mc I,\alpha\in\mc A$, such that
\begin{equation}\label{20130611:eq2}
\begin{array}{l}
\displaystyle{
\vphantom{\Big(}
\xi=B^\alpha_nF^\alpha_n
\,\,\,\,\text{ for all } \alpha \in\mc A
\,,} \\
\displaystyle{
\vphantom{\Big(}
A^\alpha_iF^\alpha_i
=B^\alpha_{i-1}F^\alpha_{i-1}
\,\,\,\,\text{ for all } 1\neq i\in\mc I\,,\alpha \in\mc A
\,,} \\
\displaystyle{
\vphantom{\Big(}
\sum_{\alpha\in\mc A} A^\alpha_1F^\alpha_1=P
\,.}
\end{array}
\end{equation}
In this case, we say that the collection $\{F^\alpha_i\}_{i,\alpha}$
is a \emph{solution} for the association relation \eqref{20130703:eq1}
over the field $\mc K_1$.
\end{definition}

In particular, Definition \ref{20130703:def1} can be rephrased by saying that
$\xi\ass{H}P$
if and only if $\xi\assk{\{A,B\}}{\mc K}P$
for some fractional decomposition $H=AB^{-1}$.
In the remainder of the section we want to establish
a deeper connection between Definition \ref{20130703:def1} and Definition \ref{20130703:def2}.
In fact, in Section \ref{sec:4.2} we prove that,
if \eqref{20130611:eq1} is a \emph{minimal} rational expression for $H$, 
then the association relation \eqref{20130703:eq1}
holds over any differential field extension of $\mc K$
if and only if $\xi\ass{H}P$.

\subsection{An upper and lower bound for the singular degree
}\label{sec:4.2a}

Let $H\in\Mat_{\ell\times\ell}\mc K(\partial)$ be a rational matrix pseudodifferential operator,
and let \eqref{20130611:eq1} be a rational expression for $H$,
with $A^\alpha_i,B^\alpha_i\in\mc K[\partial]$,
and $B^\alpha_i$ non-degenerate, for all $\alpha\in\mc A,\,i\in\mc I$.
We associate to this rational expression the following vector space,
of solutions for the zero association relation:
\begin{equation}\label{20130801:eq2}
\begin{array}{l}
\displaystyle{
\vphantom{\Big)}
\mc E:=\mc E(\{A^\alpha_i,B^\alpha_i\}_{i\in\mc I,\alpha\in\mc A})
} \\
\displaystyle{
\vphantom{\Big)}
=\Big\{(F^\alpha_i)_{i\in\mc I,\alpha\in\mc A}\in\mc L^{\ell Nn}
\text{ solution for } 
0\assk{\{A^\alpha_i,B^\alpha_i\}_{i,\alpha}}{\mc L} 0
\Big\}
\,.}
\end{array}
\end{equation}
Note that a rational expression for $H^*$ is
$$
H^*=\sum_{\alpha\in\mc A}
1({B^\alpha_n}^*)^{-1}{A^{\alpha}_n}^*
\dots
({B^\alpha_1}^*)^{-1}{A^{\alpha}_1}^*1^{-1}
\,.
$$
The corresponding vector space of solutions for the zero association relation is
\begin{equation}\label{20130801:eq3}
\begin{array}{l}
\displaystyle{
\vphantom{\Big)}
\mc E^*:=\mc E(\{{A^\alpha_{n+1-i}}^*,{B^\alpha_{n-i}}^*\}_{i\in\{0,\dots,n\},\alpha\in\mc A})
} \\
\displaystyle{
\vphantom{\Big)}
=\Big\{(F^\alpha_i)_{i\in\mc I,\alpha\in\mc A}\in\mc L^{\ell Nn}
\text{ solution for } 
0\assk{\{{A^\alpha_{n+1-i}}^*,{B^\alpha_{n-i}}^*\}_{i,\alpha}}{\mc L} 0
\Big\}
\,,}
\end{array}
\end{equation}
where we let $A^\alpha_{n+1}=B^\alpha_0=\id$.
\begin{theorem}\label{20130801:thm}
For the rational matrix pseudodifferential operator $H$,
given by the rational expression \eqref{20130611:eq1},
we have
\begin{equation}\label{20130801:eq4}
\begin{array}{l}
\displaystyle{
\vphantom{\Big)}
\sum_{i\in\mc I,\alpha\in\mc A}\deg(B^\alpha_i)
-\dim_{\bar{\mc C}}\mc E-\dim_{\bar{\mc C}}\mc E^*
\leq\sdeg(H)
} \\
\displaystyle{
\vphantom{\Big)}
\leq
\sum_{i\in\mc I,\alpha\in\mc A}\deg(B^\alpha_i)-
\max\big\{
\dim_{\bar{\mc C}}\mc E,\,
\dim_{\bar{\mc C}}\mc E^*
\big\}
\,.}
\end{array}
\end{equation}
\end{theorem}
\begin{proof}
We prove the inequalities \eqref{20130801:eq4}
for the rational expression \eqref{20130611:eq1}
by induction on the pair $(n,N)$, in lexicographic order.
For $n=N=1$ the rational expression \eqref{20130611:eq1}
reduces to $H=AB^{-1}$,
and in this case the spaces \eqref{20130801:eq2} and \eqref{20130801:eq3}
are
$$
\mc E=\big\{F\in\mc L^\ell\,\text{ solution of }\,0\assk{(A,B)}{\mc L}0\big\}
=\ker{}_{\mc L}A\cap\ker{}_{\mc L}B
\,,
$$
and
$$
\mc E^*=\big\{F\in\mc L^\ell\,\text{ solution of }\,0\assk{\{(1,B^*),(A^*,1)\}}{\mc L}0\big\}
=0
\,.
$$
Therefore, the upper and the lower bounds in \eqref{20130801:eq4}
coincide with $\deg(B)-\dim_{\bar{\mc C}}(\ker_{\mc L}A\cap\ker_{\mc L}B)$,
which is equal to $\sdeg(H)$ by Proposition \ref{20130610:lem3}(a).

Next, let us consider the case when $n=1$ and $N\geq2$.
In this case the rational expression \eqref{20130611:eq1}
becomes
\begin{equation}\label{20130801:eq5}
H=A^1(B^1)^{-1}+\dots+A^N(B^N)^{-1}
\,.
\end{equation}
In this case the spaces $\mc E$ and $\mc E^*$ 
defined in equations \eqref{20130801:eq2} and \eqref{20130801:eq3}
are, respectively,
\begin{equation}\label{20130801:eq13}
\mc E
=\Bigg\{(F^\alpha)_{\alpha=1}^{N}
\,\Bigg|\,
\begin{array}{l}
B^1F^1=\dots=B^{N}F^N=0\,,\\
A^1F^1+\dots+A^NF^N=0\,.
\end{array}
\Bigg\}
\,,
\end{equation}
and
\begin{equation}\label{20130801:eq14}
\mc E^*
=\Bigg\{(F^\alpha)_{\alpha=1}^N
\,\Bigg|\,
\begin{array}{l}
{B^1}^*F^1=\dots={B^N}^*F^N=0\,,\\
F^1+\dots+F^N=0\,.
\end{array}
\Bigg\}
\,.
\end{equation}
Let $Q\in\Mat_{\ell\times\ell}\mc K[\partial]$
be the left greatest common divisor of $B_{N-1}$ and $B_N$,
so that
\begin{equation}\label{20130801:eq6}
B^{N-1}=Q\bar{B}^{N-1}
\,\,,\,\,\,\,
B^N=Q\bar{B}^N
\,,
\end{equation}
and $\bar{B}^{N-1}$ and $\bar{B}^N$ are left coprime.
Let also
\begin{equation}\label{20130801:eq7}
\bar{B}^{N-1}C^{N-1}
=
\bar{B}^NC^N
\end{equation}
be the right least common multiple of $\bar{B}^{N-1}$ and $\bar{B}^N$.
In particular, by Lemma \ref{20130703:lem2}(a),
$C^{N-1}$ and $C^N$ are non-degenerate, right coprime,
and
\begin{equation}\label{20130801:eq8}
\deg(C^{N-1})
=
\deg(\bar{B}^N)
=
\deg(B^N)-\deg(Q)
\,.
\end{equation}
Moreover,
\begin{equation}\label{20130801:eq7b}
\widetilde{B}^{N-1}=
B^{N-1}C^{N-1}
=
B^NC^N
\end{equation}
is the right least common multiple of $B^{N-1}$ and $B^N$.
By equation \eqref{20130801:eq8} we have
\begin{equation}\label{20130801:eq12}
\deg(\widetilde{B}^{N-1})=\deg(B^{N-1})+\deg(B^N)-\deg(Q)
\,.
\end{equation}
Let also
$\widetilde{A}^{N-1}=A^{N-1}C^{N-1}+A^NC^N$.
Then, $H$ admits the following rational expression:
\begin{equation}\label{20130801:eq8b}
H=A^1(B^1)^{-1}+\dots+A^{N-2}(B^{N-2})^{-1}
+\widetilde{A}^{N-1}(\widetilde{B}^{N-1})^{-1}
\,.
\end{equation}
This rational expression has $N-1$ summands,
therefore we can apply the inductive assumption.
We have:
\begin{equation}\label{20130801:eq9}
\begin{array}{l}
\displaystyle{
\vphantom{\Big)}
\sum_{\alpha=1}^{N-2}\deg(B^\alpha)
+\deg(\widetilde{B}^{N-1})
-\dim_{\bar{\mc C}}\mc E_1-\dim_{\bar{\mc C}}\mc E_1^*
\leq\sdeg(H)
} \\
\displaystyle{
\vphantom{\Big)}
\leq
\sum_{\alpha=1}^{N-2}\deg(B^\alpha)+\deg(\widetilde{B}^{N-1})
-\max\big\{
\dim_{\bar{\mc C}}\mc E_1,\,
\dim_{\bar{\mc C}}\mc E_1^*
\big\}
\,,}
\end{array}
\end{equation}
where, recalling \eqref{20130801:eq2} and \eqref{20130801:eq3}, we let
\begin{equation}\label{20130801:eq10}
\mc E_1
=\Bigg\{(G^\alpha)_{\alpha=1}^{N-1}
\,\Bigg|\,
\begin{array}{l}
B^1G^1=\dots=B^{N-2}G^{N-2}=\widetilde{B}^{N-1}G^{N-1}=0\,,\\
A^1G^1+\dots+A^{N-2}G^{N-2}+\widetilde{A}^{N-1}G^{N-1}=0\,.
\end{array}
\Bigg\}
\,,
\end{equation}
and
\begin{equation}\label{20130801:eq11}
\mc E_1^*
=\Bigg\{(G^\alpha)_{\alpha=1}^{N-1}
\,\Bigg|\,
\begin{array}{l}
{B^1}^*G^1=\dots={B^{N-2}}^*G^{N-2}={\widetilde{B}^{N-1}}{}^*G^{N-1}=0\,,\\
G^1+\dots+G^{N-2}+G^{N-1}=0\,.
\end{array}
\Bigg\}
\,.
\end{equation}

In order to continue the proof, we need the following two lemmas.
\begin{lemma}\label{20130801:lem1}
We have an exact sequence
\begin{equation}\label{20130801:eq15}
0\to\mc E_1\stackrel{f}{\longrightarrow}\mc E\stackrel{g}{\longrightarrow}\ker{}_{\mc L}Q\,,
\end{equation}
where $f$ is the map
\begin{equation}\label{20130801:eq16}
f:\,(G^\alpha)_{\alpha=1}^{N-1}\mapsto(G^1,\dots,G^{N-2},C^{N-1}G^{N-1},C^NG^{N-1})
\,,
\end{equation}
and $g$ is the map
\begin{equation}\label{20130801:eq17}
g:\,(F^\alpha)_{\alpha=1}^N\mapsto\bar{B}^{N-1}F^{N-1}-\bar{B}^NF^N
\,.
\end{equation}
\end{lemma}
\begin{proof}
First, it is clear that the image of $g$ lies in the kernel of $Q$,
since, for $(F^\alpha)_{\alpha=1}^N\in\mc E$, we have
$$
Qg(F^\alpha)_{\alpha=1}^N=Q(\bar{B}^{N-1}F^{N-1}-\bar{B}^NF^N)=
{B}^{N-1}F^{N-1}-{B}^NF^N=0-0=0\,.
$$
Moreover, since $C^{N-1}$ and $C^N$ are right coprime,
we have, by Theorem \ref{th:6.4}(a), that $\ker_{\mc L}(C^{N-1})\cap\ker_{\mc L}(C^N)=0$.
This clearly implies that the map $f$ is injective.
We are left to prove that $\im(f)=\ker(g)$.
We have
$$
g(f(G^\alpha)_{\alpha=1}^{N-1})
=
\bar{B}^{N-1}C^{N-1}G^{N-1}-\bar{B}^NC^NG^{N-1}
=0\,,
$$
by \eqref{20130801:eq7}.
Hence, $\im(f)\subset\ker(g)$.
To prove the opposite inclusion,
let $(F^\alpha)_{\alpha=1}^N\in\ker(g)$,
i.e. 
\begin{equation}\label{20130801:eq18}
\begin{array}{l}
B^1F^1=\dots=B^{N-2}F^{N-2}=0
\,,\,\,
\bar{B}^{N-1}F^{N-1}=\bar{B}^NF^N\in\ker Q
\,,\\
A^1F^1+\dots+A^NF^N=0
\,.
\end{array}
\end{equation}
Since $\bar{B}^{N-1}$ and $\bar{B}^N$ are left coprime,
by Theorem \ref{CR13}
there exists $G^{N-1}\in\mc L^\ell$ such that
\begin{equation}\label{20130801:eq19}
F^{N-1}=C^{N-1}G^{N-1}
\,\text{ and }\,
F^N=C^NG^{N-1}
\,.
\end{equation}
Therefore, by \eqref{20130801:eq18} we have $(F^1,\dots,F^{N-2},G^N)\in\mc E_1$,
and by \eqref{20130801:eq19} we have 
$(F^\alpha)_{\alpha=1}^N=f(F^1,\dots,F^{N-2},G^N)$.
Therefore, $\ker(g)\subset\im(f)$.
\end{proof}
\begin{lemma}\label{20130801:lem2}
We have a short exact sequence
\begin{equation}\label{20130801:eq15b}
0\to\ker(Q^*)\stackrel{g^*}{\longrightarrow}
\mc E^*\stackrel{f^*}{\longrightarrow}
\mc E_1^*\to0
\,,
\end{equation}
where $f^*$ is the map
\begin{equation}\label{20130801:eq16b}
f^*:\,(F^\alpha)_{\alpha=1}^{N}\mapsto(F^1,\dots,F^{N-2},F^{N-1}+F^N)
\,,
\end{equation}
and $g^*$ is the map
\begin{equation}\label{20130801:eq17b}
g^*:\,G\mapsto(0,\dots,0,G,-G)
\,.
\end{equation}
\end{lemma}
\begin{proof}
The map $g^*$ is obviously injective,
and its image lies in $\mc E^*$, since ${B^{N-1}}^*$ and ${B^N}^*$
are divisible on the right by $Q^*$.
Moreover, since $f^*\circ g^*=0$, we have the inclusion $\im(g^*)\subset\ker(f^*)$.
The opposite inclusion is clear too:
if $(F^\alpha)_{\alpha=1}^N\in\ker f^*$,
then $F^1=\dots=F^{N-2}=0$,
and 
$$
F^N=-F^{N-1}\in\ker{}_{\mc L}({B^N}^*)\cap\ker{}_{\mc L}({B^{N-1}}^*)=\ker{}_{\mc L}(Q^*)
\,,
$$
so that $(F^\alpha)_{\alpha=1}^N=g^*(F^{N-1})$.
We are left to prove that $f^*$ is surjective.
Let $(G^\alpha)_{\alpha=1}^{N-1}\in\mc E_1^*$.
We have
\begin{equation}\label{20130801:eq20}
0={\widetilde{B}^{N-1}}{}^*G^{N-1}
={C^{N-1}}^*{B^{N-1}}^*G^{N-1}
={C^{N}}^*{B^{N}}^*G^{N-1}
\,.
\end{equation}
Recall that ${C^{N-1}}^*$ and ${C^N}^*$ are left coprime,
and their right least common multiple is
${C^{N-1}}^*{\bar{B}^{N-1}}{}^*={C^{N}}^*{\bar{B}^{N}}{}^*$.
By equation \eqref{20130801:eq20} we have, in particular, that
${C^{N-1}}^*({B^{N-1}}^*G^{N-1})={C^N}^*(0)$.
Therefore, by Theorem \ref{CR13}, there exists
$Z\in\mc L^\ell$ such that
\begin{equation}\label{20130801:eq21}
{B^{N-1}}^*G^{N-1}={\bar{B}^{N-1}}{}^*Z
\,\text{ and }\,
{\bar{B}^{N}}{}^*Z=0
\,.
\end{equation}
Note that $Q^*$ is a non-degenerate matrix,
therefore, since $\mc L$ is linearly closed, there exists $X\in\mc L^\ell$
such that $Z=Q^*X$.
It thus follows by \eqref{20130801:eq21} that
\begin{equation}\label{20130801:eq22}
{B^{N-1}}^*G^{N-1}={B^{N-1}}^*X
\,\text{ and }\,
{B^{N}}^*X=0
\,.
\end{equation}
In other words, $X\in\ker({B^{N}}^*)$ and $G^{N-1}-X\in\ker({B^{N-1}}^*)$.
But then $(G^\alpha)_{\alpha=1}^{N-1}=f^*(G^1,\dots,G^{N-2},G^{N-1}-X,X)$.
\end{proof}
By Lemma \ref{20130801:lem1} we have
\begin{equation}\label{20130801:eq23}
\dim_{\bar{\mc C}}\mc E_1\leq\dim_{\bar{\mc C}}\mc E\leq\dim_{\bar{\mc C}}\mc E_1+\deg(Q)\,,
\end{equation}
while
by Lemma \ref{20130801:lem2} we have
\begin{equation}\label{20130801:eq24}
\dim_{\bar{\mc C}}\mc E^*=\dim_{\bar{\mc C}}\mc E_1^*+\deg(Q)\,.
\end{equation}
Combining equation \eqref{20130801:eq9}
with equations \eqref{20130801:eq12}, \eqref{20130801:eq23} and \eqref{20130801:eq24},
we get \eqref{20130801:eq4}, in this case.

Next, we prove the claim in the general case, when $n\geq2$.
Recall the definition \eqref{20130801:eq2} and \eqref{20130801:eq3}
of the spaces $\mc E$ and $\mc E^*$,
which can be rewritten as follows
\begin{equation}\label{20130801:eq13b}
\mc E
=\Bigg\{(F^\alpha_i)_{i\in\mc I,\alpha\in\mc A}
\,\Bigg|\,
\begin{array}{l}
B^\alpha_nF^\alpha_n=0\,,\,\,\alpha\in\mc A\,,\\
A^\alpha_iF^\alpha_i=B^\alpha_{i-1}F^\alpha_{i-1}
\,,\,\, 2\leq i\leq n,\alpha\in\mc A\,,\\
\sum_{\alpha\in\mc A}A^\alpha_1F^\alpha_1=0\,.
\end{array}
\Bigg\}
\,,
\end{equation}
and
\begin{equation}\label{20130801:eq14b}
\mc E^*
=\Bigg\{(F^\alpha_i)_{i\in\mc A,\alpha\in\mc A}
\,\Bigg|\,
\begin{array}{l}
{B^\alpha_1}^*F^\alpha_1=0\,,\,\,\alpha\in\mc A\,,\\
{A^\alpha_i}^*F^\alpha_{i-1}={B^\alpha_i}^*F^\alpha_i
\,,\,\, 2\leq i\leq n,\alpha\in\mc A\,,\\
\sum_{\alpha\in\mc A}F^\alpha_n=0\,.
\end{array}
\Bigg\}
\,.
\end{equation}
For every $\alpha\in\mc A$, let $Q^\alpha\in\Mat_{\ell\times\ell}\mc K[\partial]$
be the left greatest common divisor of $B^\alpha_{n-1}$ and $A^\alpha_n$,
so that
\begin{equation}\label{20130801:eq6b}
B^\alpha_{n-1}=Q^\alpha\bar{B}^\alpha_{n-1}
\,\,,\,\,\,\,
A^\alpha_n=Q^\alpha\bar{A}^\alpha_n
\,,
\end{equation}
and $\bar{B}^\alpha_{n-1}$ and $\bar{A}^\alpha_n$ are left coprime.
Let also
\begin{equation}\label{20130801:eq7c}
\bar{B}^\alpha_{n-1}C^\alpha
=
\bar{A}^\alpha_nD^\alpha
\end{equation}
be the right least common multiple of $\bar{B}^\alpha_{n-1}$ and $\bar{A}^\alpha_n$.
In particular, by Lemma \ref{20130703:lem2}(a),
$C^\alpha$ and $D^\alpha$ are right coprime,
and $D^\alpha$ is non-degenerate of degree
\begin{equation}\label{20130802:eq1}
\deg(D^\alpha)=\deg(\bar{B}^\alpha_{n-1})=\deg(B^\alpha_{n-1})-\deg(Q^\alpha)\,.
\end{equation}
In view of equations \eqref{20130801:eq6b} and \eqref{20130801:eq7c},
we can rewrite the rational expression \eqref{20130611:eq1} for $H$
as follows:
\begin{equation}\label{20130801:eq8c}
H=\sum_{\alpha\in\mc A}
A^{\alpha}_1(B^\alpha_1)^{-1}\dots
A^{\alpha}_{n-2}(B^\alpha_{n-2})^{-1}(A^\alpha_{n-1}C^\alpha)(B^\alpha_nD^\alpha)^{-1}
\,.
\end{equation}
This expression has $n-1$ factors in each summand,
therefore we can apply the inductive assumption.
We have, by the inductive assumption and equation \eqref{20130802:eq1}:
\begin{equation}\label{20130801:eq9b}
\begin{array}{l}
\displaystyle{
\vphantom{\Big)}
\sum_{i\in\mc I,\alpha\in\mc A}\deg(B^\alpha_i)-\sum_{\alpha\in\mc A}\deg(Q^\alpha)
-\dim_{\bar{\mc C}}\mc E_2-\dim_{\bar{\mc C}}\mc E_2^*
\leq\sdeg(H)
} \\
\displaystyle{
\vphantom{\Big)}
\leq
\sum_{i\in\mc I,\alpha\in\mc A}\deg(B^\alpha)-\sum_{\alpha\in\mc A}\deg(Q^\alpha)
-\max\big\{
\dim_{\bar{\mc C}}\mc E_2,\,
\dim_{\bar{\mc C}}\mc E_2^*
\big\}
\,,}
\end{array}
\end{equation}
where
\begin{equation}\label{20130801:eq10b}
\mc E_2
=\left\{
(G^\alpha_i)_{1\leq i\leq n-1,\alpha\in\mc A}
\,\left|\,
\begin{array}{l}
B^\alpha_nD^\alpha G^\alpha_{n-1}=0\,,\,\,\alpha\in\mc A\,,\\
A^\alpha_{n-1}C^\alpha G^\alpha_{n-1}=B^\alpha_{n-2}G^\alpha_{n-2}\,,\,\,\alpha\in\mc A\,,\\
A^\alpha_i G^\alpha_i=B^\alpha_{i-1}G^\alpha_{i-1}
\,,\,\, 2\leq i\leq n-2,\alpha\in\mc A\,,\\
\sum_{\alpha\in\mc A}A^\alpha_1G^\alpha_1=0\,.
\end{array}
\right.\right\}
\,,
\end{equation}
and
\begin{equation}\label{20130801:eq11b}
\mc E_2^*
=\left\{
(G^\alpha_i)_{1\leq i\leq n-1,\alpha\in\mc A}
\,\left|\,
\begin{array}{l}
{B^\alpha_1}^*G^\alpha_1=0\,,\,\,\alpha\in\mc A\,,\\
{A^\alpha_i}^*G^\alpha_{i-1}={B^\alpha_i}^*G^\alpha_i
\,,\,\, 2\leq i\leq n-2,\alpha\in\mc A\,,\\
{C^\alpha}^*{A^\alpha_{n-1}}^*G^\alpha_{n-2}={D^\alpha}^*{B^\alpha_n}^*G^\alpha_{n-1}
\,,\,\,\alpha\in\mc A\,,\\
\sum_{\alpha\in\mc A}G^\alpha_{n-1}=0\,.
\end{array}
\right.\right\}
\,.
\end{equation}

In order to complete the proof, we need the following two lemmas.
\begin{lemma}\label{20130801:lem1b}
We have an exact sequence
\begin{equation}\label{20130801:eq15c}
0\to\mc E_2\stackrel{f}{\longrightarrow}\mc E\stackrel{g}{\longrightarrow}
\bigoplus_{\alpha\in\mc A}\ker{}_{\mc L}Q^\alpha\,,
\end{equation}
where $f$ is the map
\begin{equation}\label{20130801:eq16c}
f:\,(G^\alpha_i)_{1\leq i\leq n-1,\alpha\in\mc A}\mapsto
(G^\alpha_1,\dots,G^\alpha_{n-2},C^\alpha G^\alpha_{n-1},D^\alpha G^\alpha_{n-1})_{\alpha\in\mc A}
\,,
\end{equation}
and $g$ is the map
\begin{equation}\label{20130801:eq17c}
g:\,(F^\alpha_i)_{i\in\mc I,\alpha\in\mc A}\mapsto
(\bar{B}^\alpha_{n-1}F^\alpha_{n-1}-\bar{A}^\alpha_nF^\alpha_n)_{\alpha\in\mc A}
\,.
\end{equation}
\end{lemma}
\begin{proof}
First, it is clear by \eqref{20130801:eq6b} and the definition of $\mc E$
that the image of $g$ lies in $\bigoplus_{\alpha\in\mc A}\ker_{\mc L}Q^\alpha$.
Moreover, since $C^\alpha$ and $D^\alpha$ are right coprime,
$f$ is clearly injective.
The inclusion $\im(f)\subset\ker(g)$
immediately follows by the definitions \eqref{20130801:eq13b} of $\mc E$ 
and \eqref{20130801:eq10b} of $\mc E_2$,
and by equations \eqref{20130801:eq6b} and \eqref{20130801:eq7c}.
We are left to prove that $\ker(g)\subset\im(f)$.
Let $(F^\alpha_i)_{i\in\mc I,\alpha\in\mc A}\in\ker(g)$,
i.e.,
\begin{equation}\label{20130801:eq18c}
\begin{array}{l}
\displaystyle{
\vphantom{\Big(}
B^\alpha_nF^\alpha_n=0
\,,\,\,\alpha\in\mc A
\,,}\\
\displaystyle{
\vphantom{\Big(}
\bar{A}^\alpha_nF^\alpha_n=\bar{B}^\alpha_{n-1}F^\alpha_{n-1}
\,,\,\,\alpha\in\mc A
\,,}\\
\displaystyle{
\vphantom{\Big(}
A^\alpha_iF^\alpha_i=B^\alpha_{i-1}F^\alpha_{i-1}
\,,\,\,2\leq i\leq n-1,\alpha\in\mc A
\,,}\\
\displaystyle{
\vphantom{\Big(}
\sum_{\alpha\in\mc A}A^\alpha_1F^\alpha_1=0
\,.}
\end{array}
\end{equation}
Since $\bar{B}^\alpha_{n-1}$ and $\bar{A}^\alpha_n$ are left coprime,
by Theorem \ref{CR13}
there exists $G^\alpha\in\mc L^\ell$ such that
\begin{equation}\label{20130801:eq19c}
F^\alpha_{n-1}=C^\alpha G^\alpha
\,\text{ and }\,
F^\alpha_n=D^\alpha G^\alpha
\,.
\end{equation}
Therefore, by \eqref{20130801:eq18c} we have 
$(F^\alpha_1,\dots,F^\alpha_{n-2},G^\alpha)_{\alpha\in\mc A}\in\mc E_2$,
and by \eqref{20130801:eq19c} we have 
$$
(F^\alpha_i)_{i\in\mc I,\alpha\in\mc A}=f\big((F^\alpha_1,\dots,F^\alpha_{n-2},G^\alpha)_{\alpha\in\mc A}\big)
\,.
$$
\end{proof}
\begin{lemma}\label{20130801:lem2b}
We have a short exact sequence
\begin{equation}\label{20130801:eq15d}
0\to
\bigoplus_{\alpha\in\mc A}\ker{}_{\mc L}{Q^\alpha}^*\stackrel{g^*}{\longrightarrow}
\mc E^*\stackrel{f^*}{\longrightarrow}
\mc E_2^*\to0
\,,
\end{equation}
where $f^*$ is the map
\begin{equation}\label{20130801:eq16d}
f^*:\,(F^\alpha_i)_{i\in\mc I,\alpha\in\mc A}\mapsto
(F^\alpha_1,\dots,F^\alpha_{n-2},F^\alpha_n)_{\alpha\in\mc A}
\,,
\end{equation}
and $g^*$ is the map
\begin{equation}\label{20130801:eq17d}
g^*:\,(G^\alpha)_{\alpha\in\mc A}\mapsto(0,\dots,0,G^\alpha,0)_{\alpha\in\mc A}
\,.
\end{equation}
\end{lemma}
\begin{proof}
The map $g^*$ is obviously injective.
Its image lies in $\mc E^*$, since ${B^\alpha_{n-1}}^*$ and ${A^\alpha_n}^*$
are divisible on the right by ${Q^\alpha}^*$.
The inclusion $\im(g^*)\subset\ker(f^*)$ is obvious,
and the opposite inclusion $\im(g^*)\subset\ker(f^*)$ follows immediately 
from the definition of $\mc E^*$.
We are left to prove that $f^*$ is surjectve.
Let then 
$(G^\alpha_i)_{1\leq i\leq n-1,\alpha\in\mc A}\in\mc E_2^*$.
We have, in particular,
\begin{equation}\label{20130801:eq20b}
{C^\alpha}^*{A^\alpha_{n-1}}^*G^\alpha_{n-2}={D^\alpha}^*{B^\alpha_n}^*G^\alpha_{n-1}
\,,\,\,\alpha\in\mc A\,.
\end{equation}
Recall that ${C^\alpha}^*$ and ${D^\alpha}^*$
are left coprime,
and (cf. \eqref{20130801:eq7b})
${C^\alpha}^*{\bar{B}^\alpha_{n-1}}{}^*={D^\alpha}^*{\bar{A}^\alpha_{n}}{}^*$
is their right least common multiple.
Therefore, by Theorem \ref{CR13}, there exists
$Z^\alpha\in\mc L^\ell$ such that
\begin{equation}\label{20130801:eq21b}
{A^\alpha_{n-1}}^*G^\alpha_{n-2}={\bar{B}^\alpha_{n-1}}{}^*Z^\alpha
\,\text{ and }\,
{B^\alpha_n}^*G^\alpha_{n-1}={\bar{A}^\alpha_{n}}{}^*Z^\alpha
\,.
\end{equation}
Since ${Q^\alpha}^*$ is non-degenerate,
there exists $X^\alpha\in\mc L^\ell$
such that $Z^\alpha={Q^\alpha}^*X^\alpha$.
Hence, equation \eqref{20130801:eq21b} can be rewritten as
\begin{equation}\label{20130801:eq22b}
{A^\alpha_{n-1}}^*G^\alpha_{n-2}={B^\alpha_{n-1}}^*X^\alpha
\,\text{ and }\,
{B^\alpha_n}^*G^\alpha_{n-1}={A^\alpha_{n}}^*X^\alpha
\,.
\end{equation}
Equation \eqref{20130801:eq22b} guarantees that
$(G^\alpha_1,\dots,G^\alpha_{n-2},X^\alpha,G^\alpha_{n-1})_{\alpha\in\mc A}$
lies in $\mc E^*$, and, clearly,
$$
(G^\alpha_i)_{1\leq i\leq n-1,\alpha\in\mc A}
=f^*\big((G^\alpha_1,\dots,G^\alpha_{n-2},X^\alpha,G^\alpha_{n-1})_{\alpha\in\mc A}\big)
\,.
$$
\end{proof}
By Lemma \ref{20130801:lem1b} we have
\begin{equation}\label{20130801:eq23b}
\dim_{\bar{\mc C}}\mc E_2\leq\dim_{\bar{\mc C}}\mc E
\leq\dim_{\bar{\mc C}}\mc E_2+\sum_{\alpha\in\mc A}\deg(Q^\alpha)\,,
\end{equation}
while
by Lemma \ref{20130801:lem2b} we have
\begin{equation}\label{20130801:eq24b}
\dim_{\bar{\mc C}}\mc E^*=\dim_{\bar{\mc C}}\mc E_2^*+\sum_{\alpha\in\mc A}\deg(Q^\alpha)\,.
\end{equation}
Combining equation \eqref{20130801:eq9b}
with equations \eqref{20130801:eq23b} and \eqref{20130801:eq24b},
we get \eqref{20130801:eq4}.
\end{proof}

\subsection{Minimal rational expression}\label{sec:4.2}

\begin{lemma}\label{20130703:lem1}
Let $H\in\Mat_{\ell\times\ell}\mc K(\partial)$
be a rational matrix pseudodifferential operator,
with a rational expression of the form \eqref{20130611:eq1}.
Then
$$
\sdeg(H)\leq\sum_{i\in\mc I,\alpha\in\mc A}\deg(B^\alpha_i)\,.
$$
\end{lemma}
\begin{proof}
It follows immediately from Proposition \ref{20130703:lem3}.
\end{proof}

\begin{definition}\label{20130611:defa}
We say that a rational expression \eqref{20130611:eq1} is \emph{minimal} 
if 
$$
\sdeg(H)
=
\sum_{i,\alpha}\deg(B^\alpha_i)
\,.
$$
\end{definition}
\begin{corollary}\label{20130801:cor}
A rational expression \eqref{20130611:eq1}
for a rational matrix pseudodifferential operator $H\in\Mat_{\ell\times\ell}\mc K(\partial)$
is minimal if and only if $\mc E=\mc E^*=0$
(cf. equations \eqref{20130801:eq2} and \eqref{20130801:eq3}).
\end{corollary}
\begin{proof}
It follows immediately from Theorem \ref{20130801:thm}.
\end{proof}

\subsection{The main results on the association relation}\label{sec:4.2b}

\begin{theorem}\label{20130704:prop}
Let $H\in\Mat_{\ell\times\ell}\mc K(\partial)$ and let $\xi,P\in\mc K^\ell$. Then
\begin{enumerate}[(a)]
\item
The association relation
\begin{equation}\label{20130802:eq2}
\xi\assk{\{A^\alpha_i,B^\alpha_i\}_{i,\alpha}}{\mc K_1}P
\,,
\end{equation}
is independent of the minimal rational expression \eqref{20130611:eq1} for $H$
and of the intermediate differential field $\mc K\subset\mc K_1\subset\mc L$.
In particular, it is equivalent to $\xi\ass{H}P$.
\item
If $\xi\ass{H}P$, then the association relation
\begin{equation}\label{20130802:eq2}
\xi\assk{\{A^\alpha_i,B^\alpha_i\}_{i,\alpha}}{\mc L}P
\,,
\end{equation}
holds for any rational expression \eqref{20130611:eq1} for $H$.
\end{enumerate}
\end{theorem}
The rest of this section will be devoted to the proof of Theorem \ref{20130704:prop}.
\begin{lemma}\label{20130802:prop4}
Let $H\in\Mat_{\ell\times\ell}\mc K(\partial)$
and let $H=AB^{-1}$ be a minimal fractional decomposition for $H$.
Then, for every $\xi,P\in\mc K^\ell$,
$$
\xi\ass{H}P
\,\text{ if and only if }\,
\xi\assk{\{A,B\}}{\mc K}P
\,.
$$
\end{lemma}
\begin{proof}
The ``if'' part is obvious.
Recall that, by definition, $\xi\ass{H}P$
if and only if there exists a fractional decomposition 
$H=\widetilde{A}\widetilde{B}^{-1}$ for $H$
such that $\xi\assk{\{\widetilde{A},\widetilde{B}\}}{\mc K}P$.
On the other hand,
by Theorem \ref{th:6.4}(b)
there exists a non-degenerate matrix $D\in\Mat_{\ell\times\ell}\mc K[\partial]$
such that $\widetilde{A}=AD$ and $\widetilde{B}=BD$.
Therefore, if $F\in\mc K^\ell$ is a solution for 
the association relation $\xi\assk{\{\widetilde{A},\widetilde{B}\}}{\mc K}P$,
then $DF$ is a solution for $\xi\assk{\{A,B\}}{\mc K}P$.
\end{proof}
\begin{lemma}\label{20130802:prop3}
Let $H\in\Mat_{\ell\times\ell}\mc K(\partial)$,
let $H=AB^{-1}$ be a minimal fractional decomposition for $H$,
and let \eqref{20130611:eq1} be an arbitrary rational expression for $H$.
Then, for every $\xi,P\in\mc K^\ell$,
$$
\xi\assk{\{A,B\}}{\mc L}P
\,\text{ implies }\,
\xi\assk{\{A^\alpha_i,B^\alpha_i\}_{i,\alpha}}{\mc L}P
\,.
$$
\end{lemma}
\begin{proof}
Consider the rational expression \eqref{20130611:eq1}.
For every $\alpha\in\mc A$,
we can apply Lemma \ref{20130801:lem} to get matrices
$X^\alpha_1,\dots,X^\alpha_n\in\Mat_{\ell\times\ell}\mc K[\partial]$,
with $X^\alpha_n$ non-degenerate,
such that
\begin{equation}\label{20130801:eq1c}
B^\alpha_iX^\alpha_i=A^\alpha_{i+1}X^\alpha_{i+1}
\text{ for all } i=1,\dots,n-1,\,\alpha\in\mc A
\,.
\end{equation}
Then the rational matrix $H$ admits
the following new rational expression:
\begin{equation}\label{20130801:eq1d}
H=\sum_{\alpha\in\mc A}
(A^\alpha_1X^\alpha_1)(B^\alpha_nX^\alpha_n)^{-1}
\,.
\end{equation}
Next, let
\begin{equation}\label{20130802:eq3}
\widetilde{B}=B^1_nX^1_nC^1=\dots=B^N_nX^N_nC^N\,,
\end{equation}
be the least right common multiple of $B^1_nX^1_n,\dots,B^N_nX^N_n$.
We thus get the fractional decomposition $H=\widetilde{A}\widetilde{B}^{-1}$, where:
\begin{equation}\label{20130801:eq1e}
\widetilde{A}=
A^1_1X^1_1C^1+\dots+A^N_1X^N_1C^N
\,.
\end{equation}
By Theorem \ref{th:6.4}(b)
there exists a non-degenerate matrix $D\in\Mat_{\ell\times\ell}\mc K[\partial]$
such that
\begin{equation}\label{20130802:eq4}
\widetilde{A}=AD
\,\text{ and }\,
\widetilde{B}=BD
\end{equation}
By assumption, $\xi\assk{\{A,B\}}{\mc L}P$.
In other words, there exists $F\in\mc L^\ell$ such that $BF=\xi$ and $AF=P$.
Since $D$ is non-degenerate and $\mc L$ is linearly closed,
there exists $Z\in\mc L^\ell$ such that $F=DZ$.
Therefore, $Z$ is a solution for $\xi\assk{\{\widetilde{A},\widetilde{B}\}}{\mc L}P$.
It is straightforward to check, using equations 
\eqref{20130801:eq1c}, \eqref{20130802:eq3} and \eqref{20130801:eq1e},
that, letting
$Z^\alpha=C^\alpha Z,\,\alpha\in\mc A$,
we get a solution for 
$$
\xi\assk{\{A^\alpha_1X^\alpha_1,B^\alpha_nX^\alpha_n\}_{\alpha\in\mc A}}{\mc L}P
\,,
$$
and letting $Z^\alpha_i=X^\alpha_iZ^\alpha,\,i\in\mc I,\alpha\in\mc A$,
we get a solution for 
$$
\xi\assk{\{A^\alpha_i,B^\alpha_i\}_{i\in\mc I,\alpha\in\mc A}}{\mc L}P
\,.
$$
\end{proof}
\begin{lemma}\label{20130802:prop1}
Let $H\in\Mat_{\ell\times\ell}\mc K(\partial)$,
let $H=AB^{-1}$ be a minimal fractional decomposition for $H$,
and let \eqref{20130611:eq1} be a minimal rational expression for $H$.
Then, for every $\xi,P\in\mc K^\ell$, we have
$$
\xi\assk{\{A,B\}}{\mc L}P
\,\text{ if and only if }\,
\xi\assk{\{A^\alpha_i,B^\alpha_i\}_{i,\alpha}}{\mc L}P
\,.
$$
\end{lemma}
\begin{proof}
The ``only if'' part is given by Lemma \ref{20130802:prop3}, so we only need to prove the ``if'' part.
Assume 
that $\xi\assk{\{A^\alpha_i,B^\alpha_i\}_{i,\alpha}}{\mc L}P$.
We shall prove that $\xi\assk{\{A,B\}}{\mc L}P$
by induction on the ordered pair $(N,n)$.
%
%
In the case $N=n=1$ the statement is obvious since,
by Theorem \ref{th:6.4}, two minimal fractional decompositions for $H$
are obtained from each other by multiplication on the right
by an invertible $\ell\times\ell$ matrix differential operator.

%
Next, we consider the case when $N=1$ and $n\geq2$.
In this case, the rational expression \eqref{20130611:eq1} is 
\begin{equation}\label{20130704:eq3}
H=A_1B_1^{-1}\dots A_{n-1}B_{n-1}^{-1}A_nB_n^{-1}\,,
\end{equation}
and by the minimality assumption we have
\begin{equation}\label{20130704:eq1}
\sdeg(H)=\deg(B_1)+\deg(B_2)+\dots+\deg(B_n)
\,.
\end{equation}
By Lemma \ref{20130703:lem2}(a),
there exist right coprime matrix differential operators
$\widetilde{A}_n,\widetilde{B}_{n-1}\in\Mat_{\ell\times\ell}\mc K[\partial]$,
with $\widetilde{B}_{n-1}$ non-degenerate,
such that
\begin{equation}\label{20130704:eq4}
\text{right l.c.m.}(A_n,B_{n-1})=A_n\widetilde{B}_{n-1}=B_{n-1}\widetilde{A}_n\,,
\end{equation}
and, moreover, 
\begin{equation}\label{20130704:eq2}
\deg(\widetilde{B}_{n-1})=\deg(B_{n-1})
\,,
\end{equation}
since, by minimality of \eqref{20130704:eq3}
$A_n$ and $B_{n-1}$ are left coprime.
Combining equations \eqref{20130704:eq3} and \eqref{20130704:eq4},
we get the following new rational expression for $H$,
with $n-1$ factors:
\begin{equation}\label{20130704:eq5}
H=A_1B_1^{-1}\dots A_{n-2}B_{n-2}^{-1}A_{n-1}\widetilde{A}_n\big(B_n\widetilde{B}_{n-1}\big)^{-1}
\,,
\end{equation}
which is again minimal by \eqref{20130704:eq1} and \eqref{20130704:eq2}.
By the inductive assumption we have:
$$
\xi\assk{\{A_1,B_1,\dots,A_{n-2},B_{n-2},A_{n-1}\widetilde{A}_n,B_n\widetilde{B}_{n-1}\}}{\mc L}P
\,\text{ implies }\,
\xi\assk{\{A,B\}}{\mc L}P
\,,
$$
and we have to prove that
\begin{equation}\label{20130705:eq1}
\xi\assk{\{A_i,B_i\}_{i\in\mc I}}{\mc L}P
\,\text{ implies }\,
\xi\assk{\{A_1,B_1,\dots,A_{n-2},B_{n-2},A_{n-1}\widetilde{A}_n,B_n\widetilde{B}_{n-1}\}}{\mc L}P
\,.
\end{equation}
A solution for the association relation in the left of \eqref{20130705:eq1}
is an $n$-tuple $F_1,\dots,F_n\in\mc L^{\ell}$ such that
\begin{equation}\label{20130705:eq2}
\begin{array}{l}
\displaystyle{
\vphantom{\Big(}
B_nF_n=\xi
\,\,,\,\,\,\,
A_nF_n=B_{n-1}F_{n-1}
\,\,,\,\,\,\,
A_{n-1}F_{n-1}=B_{n-2}F_{n-2}
\,\,,
} \\
\displaystyle{
\vphantom{\Big(}
A_{n-2}F_{n-2}=B_{n-3}F_{n-3}
\,\,,\,\,\dots\,\,,\,\,
A_2F_2=B_1F_1
\,\,,\,\,\,\,
A_1F_1=P
\,,}
\end{array}
\end{equation}
Since $A_n$ and $B_{n-1}$ are left coprime, 
by the second identity in \eqref{20130705:eq2}
and Theorem \ref{CR13},
there exists $\widetilde{F}_{n-1}\in\mc L^{\ell}$
such that
$F_{n-1}=\widetilde{A}_n\widetilde{F}_{n-1}$
and $F_n=\widetilde{B}_{n-1}\widetilde{F}_{n-1}$.
It is then immediate to check that $F_1,\dots,F_{n-2},\widetilde{F}_{n-1}$
is a solution for the association relation in the right of \eqref{20130705:eq1}.


Next, we consider the general case when $N\geq2$.
In this case, we have
$H=H^1+\dots+H^N$, where
\begin{equation}\label{20130704:eq11}
H^\alpha=A^\alpha_1(B^\alpha_1)^{-1}\dots A^\alpha_n(B^\alpha_n)^{-1}
\,\,,\,\,\,\,
\alpha=1,\dots,N
\,.
\end{equation}
By Proposition \ref{20130703:lem3} it follows that
$$
\sdeg(H)\leq\sum_{\alpha\in\mc A}\sdeg(H^\alpha)
\leq\sum_{\alpha\in\mc A}\sum_{i\in\mc I}\deg(B^\alpha_i)
\,.
$$
Hence, 
since, by assumption, \eqref{20130611:eq1} is a minimal rational expression for $H$,
all inequalities above are in fact equalities.
In particular, \eqref{20130704:eq11} is a minimal rational expression for $H^\alpha$
for every $\alpha$, and
\begin{equation}\label{20130704:eq12}
\sdeg(H)=\sdeg(H^1)+\dots+\sdeg(H^N)
\,.
\end{equation}
For every $\alpha\in\mc A$,
let $H^\alpha=A^\alpha(B^\alpha)^{-1}$ be a minimal fractional decomposition for $H^\alpha$,
and let
\begin{equation}\label{20130704:eq13}
\widetilde{B}=B^1C^1=\dots=B^NC^N
\,,
\end{equation}
be the right least common multiple of $B^1,\dots,B^N$.
Thanks to equation \eqref{20130704:eq12},
we can apply Lemma \ref{20130705:lem3} to conclude that
the matrices $B^1,\dots,B^N$ are strongly left coprime,
and that
\begin{equation}\label{20130704:eq13b}
H=(A^1C^1+\dots A^NC^N)\widetilde{B}^{-1}
\,,
\end{equation}
is a minimal fractional decomposition for $H$.
By definition, 
$\{F^\alpha_i\}_{i\in\mc I,\alpha\in\mc A}\subset\mc L^\ell$ 
is a solution for the association relation
$\xi\assk{\{A^\alpha_i,B^\alpha_i\}_{i\in\mc I,\alpha\in\mc A}}{\mc L}P$
if and only if, for every $\alpha\in\mc A$, 
$\{F^\alpha_i\}_{i\in\mc I}$ 
is a solution for the association relation
$\xi\assk{\{A^\alpha_i,B^\alpha_i\}_{i\in\mc I}}{\mc L}A^\alpha_1F^\alpha_1=:P^\alpha$,
and $P^1+\dots+P^N=P$.
Hence,
\begin{equation}\label{20130704:eq14}
\xi\assk{\{A^\alpha_i,B^\alpha_i\}_{i\in\mc I,\alpha\in\mc A}}{\mc L}P
\,\,\text{ if and only if }\,\,
\xi\assk{\{A^\alpha_i,B^\alpha_i\}_{i\in\mc I}}{\mc L}P^\alpha \,,
\end{equation}
for some $P^1,\dots,P^N\in\mc L^\ell$ such that $P^1+\dots+P^N=P$.
On the other hand, by the case $N=1$
we have that, for every $\alpha\in\mc A$,
\begin{equation}\label{20130704:eq15}
\xi\assk{\{A^\alpha_i,B^\alpha_i\}_{i\in\mc I}}{\mc L}P^\alpha
\,\,\text{ if and only if }\,\,
\xi\assk{\{A^\alpha,B^\alpha\}}{\mc L}P^\alpha
\,.
\end{equation}
By definition, the association relation in the right of \eqref{20130704:eq15}
means that there exists $F^\alpha\in\mc L^\ell$ such that
\begin{equation}\label{20130704:eq16}
B^\alpha F^\alpha=\xi
\,\,\text{ and }\,\,
A^\alpha F^\alpha=P^\alpha
\,.
\end{equation}
Since $B^1,\dots,B^N$ are strongly left coprime,
it follows by the first equation in \eqref{20130704:eq16} and Theorem \ref{20130705:thm}
that there exists $F\in\mc L^\ell$ such that
$F^\alpha=C^\alpha F$ for every $\alpha\in\mc A$.
Hence, $\widetilde{B}F=\xi$, and, by by the second equation in \eqref{20130704:eq14}, 
$$
(A^1C^1+\dots+A^NC^N)F=A^1F^1+\dots+A^NF^N=P\,.
$$
In other words, $F\in\mc L^\ell$ is a solution for the association relation
\begin{equation}\label{20130704:eq17}
\xi\assk{\{A^1C^1+\dots+A^NC^N,\widetilde{B}\}}{\mc L}P
\,.
\end{equation}
Since any minimal fractional decompositions for $H$ differ by multiplication
by an invertible $\ell\times\ell$ matrix differential operator,
we thus get $\xi\assk{\{A,B\}}{\mc L}P$, completing the proof.
\end{proof}
\begin{lemma}\label{20130802:prop2}
Let $H\in\Mat_{\ell\times\ell}\mc K(\partial)$,
and let \eqref{20130611:eq1} be a minimal rational expression for $H$.
Then, for every $\xi,P\in\mc K^\ell$,
$$
\xi\assk{\{A^\alpha_i,B^\alpha_i\}_{i,\alpha}}{\mc L}P
\,\text{ implies }\,
\xi\assk{\{A^\alpha_i,B^\alpha_i\}_{i,\alpha}}{\mc K}P
\,.
$$
\end{lemma}
\begin{proof}
Let $\{F^\alpha_i\}_{i\in\mc I,\alpha\in\alpha}\subset\mc L$
be a solution for 
$\xi\assk{\{A^\alpha_i,B^\alpha_i\}_{i,\alpha}}{\mc L}P$.
Since, by assumption, the rational expression \eqref{20130611:eq1} is minimal,
we have from Corollary \ref{20130801:cor}
that the space $\mc E$ of solutions for the association relation
$0\assk{\{A^\alpha_i,B^\alpha_i\}_{i,\alpha}}{\mc L}0$
is zero.
Therefore, by linearity, $\{F^\alpha_i\}_{i\in\mc I,\alpha\in\alpha}\subset\mc L$
must be the unique solution for $\xi\assk{\{A^\alpha_i,B^\alpha_i\}_{i,\alpha}}{\mc L}P$.

We want to prove that, in fact, $F^\alpha_i$ lies in $\mc K^\ell$ for every $i,\alpha$.
For this, we shall use some differential Galois theory (see e.g. \cite{PS03}).
Let $\widetilde{\mc K}=\bar{\mc C}\otimes_{\mc C}\mc K$.
By \cite[Lem.5.12(a)]{CDSK13b}, $\widetilde{\mc K}$ is a differential field extension of $\mc K$,
with field of constants $\bar{\mc C}$,
and the linear closure $\mc L$ is obtained as union of  the Picard-Vessiot composita
$\widetilde{\mc K}=\widetilde{\mc K}_0\subset\widetilde{\mc K}_1\subset\dots\subset\mc L$,
see \cite{Mag94}.
Suppose that, for some $i,\alpha$, one of the entries of $F^\alpha_i$ 
does not lie in $\widetilde{\mc K}$.
Then, by \cite[Lem.5.9]{CDSK13b}, there exists $k$
and a Picard-Vessiot extension $\mc P$ of $\mc K_k$
such that, for every $i,\alpha$, all the entries of $F^\alpha_i$ 
lie in $\mc P$, and not all lie in $\mc K_k$.
Clearly, being the unique solution for the association relation
$\xi\assk{\{A^\alpha_i,B^\alpha_i\}_{i,\alpha}}{\mc L}P$
(with all the matrices $A^\alpha_i,B^\alpha_i$ with coefficients in $\mc K$),
the element $(F^\alpha_i)_{i\in\mc I,\alpha\in\mc A}\in\mc P^{\ell nN}$ 
is fixed by the differential Galois group $Gal(\mc P/\mc K_k)$.
Therefore, by \cite[Prop.5.14]{CDSK13b}
all the entries of $F^\alpha_i$ lie in $\mc K_k$, which is a contradiction.
Therefore, $F^\alpha_i\in\widetilde{\mc K}^\ell$ for every $i,\alpha$.
In order to prove that $F^\alpha_i\in\mc K^\ell$ for every $i,\alpha$,
we apply the ordinary Galoise theory.
Clearly, the entries of $F^\alpha_i$, being elements of $\bar{\mc C}\otimes_{\mc C}\mc K$,
lie in a finite Galois extension of $\mc K$.
Again, being the unique solution for the association relation
$\xi\assk{\{A^\alpha_i,B^\alpha_i\}_{i,\alpha}}{\mc L}P$,
$(F^\alpha_i)_{i\in\mc I,\alpha\in\mc A}$ is fixed by the corresponding Galois group,
and therefore all the entries lie in $\mc K$.
\end{proof}
\begin{corollary}\label{20130802:cor}
Let $H\in\Mat_{\ell\times\ell}\mc K(\partial)$,
and let \eqref{20130611:eq1} be a minimal rational expression for $H$.
Then, the association relation
$$
\xi\assk{\{A^\alpha_i,B^\alpha_i\}_{i,\alpha}}{\mc K_1}P
\,,
$$
is independent of the intermediate differential field $\mc K\subset\mc K_1\subset\mc L$.
\end{corollary}
\begin{proof}
By definition of association relation,
if $\mc K_1\subset\mc K_2$,
then 
$\xi\assk{\{A^\alpha_i,B^\alpha_i\}_{i,\alpha}}{\mc K_1}P$ 
implies
$\xi\assk{\{A^\alpha_i,B^\alpha_i\}_{i,\alpha}}{\mc K_2}P$.
Therefore the statement follows immediately from Lemma \ref{20130802:prop2}.
\end{proof}
\begin{proof}[Proof of Theorem \ref{20130704:prop}]
The first assertion of part (a) is an immediate consequence 
of Lemma \ref{20130802:prop4}, Lemma \ref{20130802:prop1},
and Corollary \ref{20130802:cor}.
Part (b) follows from part (a) and Lemma \ref{20130802:prop3}.
\end{proof}




\begin{thebibliography}{0}
\bibitem{Art57}
E. Artin, \emph{Geometric algebra},
Interscience Publishers, Inc., New York-London, 1957.

\bibitem{CDSK12}
S. Carpentier, A. De Sole, and V.G. Kac,
\emph{Some algebraic properties of differential operators}, 
J. Math. Phys. {\bf 53} (2012), no.6, 063501.

\bibitem{CDSK13a}
S. Carpentier, A. De Sole, and V.G. Kac,
\emph{Some remarks on non-commutative principal ideal rings},
Compte Rendu Math. Acad. Sci. Paris {\bf 351} (2013), no.1-2, 5-8.

\bibitem{CDSK13b}
S. Carpentier, A. De Sole, and V.G. Kac,
\emph{Rational matrix pseudodifferential operators},
to appear in Selecta Math.,
arXiv:1206.4165

\bibitem{DSK11} 
A. De Sole, and V.G. Kac,
\emph{The variational Poisson cohomology},
Jpn. J. Math. {\bf 8} (2013), no. 1, 1-145.

\bibitem{DSK13} 
A. De Sole, and V.G. Kac,
\emph{Non-local Poisson structures and applications to the theory of
integrable systems},
Jpn. J. Math. {\bf 8} (2013), no. 2, 233-347.

\bibitem{Die43}
J. Dieudonn\'{e}, \emph{Les d\'{e}terminants sur un corps non commutatif},
Bull. Soc. Math. France \textbf{71} (1943), 27--45.

\bibitem{Mag94} 
A. R. Magid, \emph{Lectures on differential Galois theory},
University Lecture Ser., vol \textbf{7}. AMS, 1994.

\bibitem{PS03}
M. van der Put, M.F. Singer,
\emph{Galois theory of linear differential equations},
Grundlehren der Mathematischen Wissenschaften 
[Fundamental Principles of Mathematical Sciences], 328. Springer-Verlag, Berlin, 2003.


\end{thebibliography}
\end{document}